\documentclass{amsart}

\usepackage{hyperref}
\usepackage[shortcuts]{extdash}
\usepackage{thmtools}
\usepackage{amssymb}
\usepackage{mathtools}
\usepackage{eucal}
\usepackage{tikz-cd}
\usepackage{cleveref}
\usetikzlibrary{positioning,fit,shapes.geometric}

\usepackage[
	natbib=true,
	style=alphabetic,
	url=false,
	backend=bibtex
]{biblatex}
\renewbibmacro{in:}{}
\makeatletter
\renewcommand\operator@font{\sf}
\makeatother

\numberwithin{equation}{section}

\declaretheorem[sibling=equation]{theorem}
\declaretheorem[sibling=equation]{proposition}
\declaretheorem[sibling=equation]{lemma}
\declaretheorem[sibling=equation]{corollary}
\declaretheorem[style=definition,sibling=equation]{definition}
\declaretheorem[style=remark,sibling=equation]{remark}
\declaretheorem[style=remark,sibling=equation]{example}

\newcommand{\T}{\mathcal{T}}
\newcommand{\A}{\mathcal{A}}
\newcommand{\C}{\mathcal{C}}
\newcommand{\Set}{\operatorname{Set}}
\newcommand{\Aut}[1]{\operatorname{Aut}(#1)}

\newcommand{\Out}[1]{\operatorname{Out}(#1)}
\newcommand{\Pic}[1]{\operatorname{Pic}(#1)}
\newcommand{\Picst}[1]{\underline{\operatorname{Pic}}(#1)}
\renewcommand{\mod}[1]{\operatorname{mod}(#1)}

\newcommand{\proj}[1]{\operatorname{proj}(#1)}
\newcommand{\modst}[1]{\underline{\operatorname{mod}}(#1)}

\newcommand{\D}[1]{D(#1)}
\newcommand{\id}[1]{\operatorname{id}_{#1}}
\newcommand{\coker}{\operatorname{coker}}

\renewcommand{\L}[1]{\mathbb{L}_{#1}}
\newcommand{\Z}{\mathbb{Z}}
\newcommand{\Q}{\mathbb{Q}}
\newcommand{\Fp}[1]{\mathbb{F}_{#1}}
\newcommand{\characteristic}{\operatorname{char}}
\newcommand{\rad}[1]{\operatorname{rad}(#1)}
\newcommand{\soc}[1]{\operatorname{soc}(#1)}

\newcommand{\HH}[3]{\operatorname{HH}^{#1}_{#3}(#2)}

\newcommand{\B}[3]{\operatorname{B}^{#3}_{#1}(#2)}
\newcommand{\nB}[3]{\bar{\operatorname{B}}^{#3}_{#1}(#2)}
\newcommand{\HC}[3]{\operatorname{C}_{#3}^{#1}(#2)}
\newcommand{\nHC}[2]{\bar{\operatorname{C}}^{#1}(#2)}

\newcommand{\HML}[2]{\operatorname{HML}^{#1}(#2)}
\newcommand{\HMLhomology}[2]{\operatorname{HML}_{#1}(#2)}
\renewcommand{\hom}{\operatorname{Hom}}
\newcommand{\homst}{\underline{\hom}}
\newcommand{\ext}{\operatorname{Ext}}

\newcommand{\edge}[1]{\operatorname{e}_{#1}}

\DeclareRobustCommand{\svdots}{
  \vbox{%
    \baselineskip=0.33333\normalbaselineskip
    \lineskiplimit=10pt
    \hbox{\footnotesize.}\hbox{\footnotesize.}\hbox{\footnotesize.}%
    \kern-0\baselineskip
  }%
}

\title[An exotic finite pretriangulated category]{An exotic finite pretriangulated category over any algebraically closed field}

\author[J.~Díaz Cabrera]{Javier Díaz Cabrera}
\author[F.~Muro]{Fernando Muro}
\address{Universidad de Sevilla,
Facultad de Matemáticas,
Departamento de Álgebra,
Calle Tarfia s/n,
41012 Sevilla, Spain}
\email[J.~Díaz Cabrera]{jdiaz7@us.es}
\email[F.~Muro]{fmuro@us.es}
\urladdr[F. Muro]{https://personal.us.es/fmuro/}

\begin{document}

\begin{abstract}
We show that, over any algebraically closed field of any characteristic, the category of finite-dimensional projective modules over the preprojective algebra of generalized Dynkin type $\mathbb{L}_2$ has a pretriangulated category structure which is neither an algebraic nor a topological triangulated category structure.
\end{abstract}

\maketitle

\tableofcontents

\section{Introduction}\label{sec:introduction}

Triangulated categories play an essential role in different branches of mathematics. They were independently introduced by Puppe \cite{Puppe_1962} and Verdier \cite{Verdier_1977} in the early 1960s. Verdier encoded an extra axiom though, the \emph{octahedral axiom}. Ever since, Verdier's axiomatic framework became standard. Puppe's notion, i.e.~not requiring the octahedral axiom, is often referred to as a \emph{pretriangulated category} \cite[Definition 1.1.2]{neeman_2001_triangulated_categories}. Here, we adopt this terminology, which already appears in the title.

Triangulated categories usually arise in nature as stable categories of Frobenius abelian categories or as homotopy categories of stable model categories, or as full triangulated subcategories of those. The first kind of triangulated categories are called \emph{algebraic} \cite{keller_2007_differential_graded_categories} and the second kind are called \emph{topological} \cite{schwede_2010_algebraic_topological_triangulated}. A (pre)triangulated category which is not an algebraic or topological triangulated category is called \emph{exotic}. Few examples of exotic (pre)triangulated categories are known.

The first family of exotic triangulated categories was discovered in \cite{muro_schwede_strickland_2007_triangulated_categories_models}. They are categories of finitely generated projective modules over $R=\Z/(4)$ and similar rings. They have a triangulated category structure with identity shift functor whose exact triangles are the direct sums of trivial triangles and copies of the following one:
\[R\stackrel{2}{\longrightarrow}R\stackrel{2}{\longrightarrow}R\stackrel{2}{\longrightarrow}R.\]

The second class of exotic triangulated categories was found in \cite{Rizzardo2020}. They are defined over a field $k$ of characteristic $0$. These examples have a semi-orthogonal decomposition $\langle\D{K},\D{R_\eta}\rangle$ where $R=k[x_1,\dots,x_n]$, $n\geq 3$, $K$ is the field of fractions of $R$, and $R_\eta$ is a $k[\varepsilon]/(\varepsilon^2)$-linear $A_\infty$-deformation of $R[\varepsilon]/(\varepsilon^2)$ with $\varepsilon$ a cocycle of degree $2-n$.

Two examples of pretriangulated categories which are not triangulated have been recently discovered \cite{Chen_Liu_Lu_Zhang_2026, Name_Name_2026}, the second one by an anonymous author. They are defined over characteristic $2$ fields. Both papers use the same techniques. The first one acknowledges the use of AI, and the second one is computer-assisted.

We here present a new class of exotic pretriangulated categories, defined over any algebraically closed field of arbitrary characteristic.

The \emph{preprojective algebra of generalized Dynkin type $\L{2}$} over a field $k$ is the algebra $P(\L{2})$ of the following quiver with relations, see \cite{bialkowski_erdmann_skowronski_2007_deformed_preprojective_algebras},
\begin{equation}\label{eq:presentation}
\begin{tikzcd}
1 \arrow[r, shift left, "a_1"] \arrow[loop right, "b"', in=210,out=150,looseness=5] & 2, \arrow[l, shift left, "a_2"]
\end{tikzcd}
\qquad
a_1a_2=b^2,\qquad a_2a_1=0.
\end{equation}
It coincides with the first \emph{penny-farthing algebra} introduced in \cite[\S13.12]{gabriel_roiter_1997_representations_finitedimensional_algebras} and, if $k$ is algebraically closed, with the stable Auslander--Reiten algebra of the simple plane curve singularity $k[[x,y]]/(x^5+y^2)$ \cite{dieterich_wiedemann_1986_auslanderreiten_quiver_simple}.

\begin{theorem}\label{thm:main}
	Let $k$ be any algebraically closed field. The category $\proj{P(\L{2})}$ of finite\-/dimensional projective right modules over $P(\L{2})$ has a pretriangulated category structure which is neither an algebraic nor a topological triangulated category structure.
\end{theorem}

The inverse shift functor $P\mapsto P_{\sigma}$ is the restriction of scalars along the following algebra automorphism:
\begin{equation}\label{eq:sigma_automorphism}
	\sigma\colon P(\L{2}) \stackrel{\cong}{\longrightarrow} P(\L{2})\colon
	e_1 \mapsto e_1,
	e_2 \mapsto e_2,
	a_1 \mapsto -a_1,
	a_2 \mapsto -a_2,
	b \mapsto -b-b^3.
\end{equation}
Here $e_1, e_2$ are the idempotents corresponding to the vertices.
Moreover, the following triangles are exact and any other exact triangle is isomorphic to a direct sum of copies of these and their translations. Here we write $P^1=e_1P(\L{2})$ and $P^2=e_2P(\L{2})$ for the indecomposable projectives.
\begin{equation}\label{eq:triangles}
	\begin{gathered}
		0\xrightarrow{}P^1\xrightarrow{\id{}}P^1\xrightarrow{} 0,\\
		0\xrightarrow{}P^2\xrightarrow{\id{}}P^2\xrightarrow{} 0,\\
		P^1_\sigma\xrightarrow{b^3}P^1\xrightarrow{\left(\begin{smallmatrix}b\\-a_2\end{smallmatrix}\right)}
		P^1\oplus P^2\xrightarrow{\left(\begin{smallmatrix}b&a_1\end{smallmatrix}\right)} P^1,\\
		P^1_\sigma\xrightarrow{a_2b}P^2\xrightarrow{ba_1} P^1\xrightarrow{b} P^1,\\
		P^1_\sigma\oplus P^2_\sigma\xrightarrow{\left(\begin{smallmatrix}-b^2&-ba_1\\a_2b&0\end{smallmatrix}\right)}P^1\oplus P^2\xrightarrow{\left(\begin{smallmatrix}-b^2&0\\b&a_1\end{smallmatrix}\right)} P^1\oplus P^1\xrightarrow{\left(\begin{smallmatrix}b&b^2\\-a_2&0\end{smallmatrix}\right)} P^1\oplus P^2,\\
		P^1_\sigma\xrightarrow{\left(\begin{smallmatrix}-b^2\\a_2b\end{smallmatrix}\right)}P^1\oplus P^2\xrightarrow{\left(\begin{smallmatrix}b&a_1\\-a_2&0\end{smallmatrix}\right)} P^1\oplus P^2\xrightarrow{\left(\begin{smallmatrix}b^2&ba_1\end{smallmatrix}\right)} P^1,\\
		P^1_{\sigma}\xrightarrow{b^2}P^1\xrightarrow{\left(\begin{smallmatrix}-a_2b\\a_2\end{smallmatrix}\right)}P^2\oplus P^2\xrightarrow{\left(\begin{smallmatrix}a_1&ba_1\end{smallmatrix}\right)}P^1,\\
		P^2_\sigma\xrightarrow{-a_2ba_1}P^2\xrightarrow{a_1}P^1\xrightarrow{a_2}P^2.
	\end{gathered}
\end{equation}

\Cref{thm:main} is proved in the following way. In \Cref{thm:non-algebraic}, we give a criterion for the existence of a non-algebraic pretriangulated category structure on the category $\proj{\Lambda}$ of finite\-/dimensional projective modules over a finite\-/dimensional algebra $\Lambda$. The obstruction is in the shift functor, i.e.~we show that there cannot be an algebraic triangulated category structure with a given shift functor. In \Cref{cor:non-algebraic} we apply this criterion to $\Lambda=P(\L{2})$. Finally, we show in \Cref{thm:non-topological} that, if $\Lambda=\Lambda^{\Z}\otimes_{\Z}k$ for some ring $\Lambda^{\Z}$ which is free as a $\Z$-module, then any non-algebraic pretriangulated category structure on $\proj{\Lambda}$ arising from \Cref{thm:non-algebraic} is not topological either. This applies to $\Lambda=P(\L{2})$, see \Cref{cor:non-topological}.

The examples in \cite{Chen_Liu_Lu_Zhang_2026, Name_Name_2026} are built upon $\proj{P(\mathbb{A}_{5})}$ and $\proj{P(\mathbb{D}_{4})}$. Despite the similarity of the underlying categories, their techniques are totally different from ours. They start with the known algebraic triangulated structures on these categories and tweak the exact triangles in order to construct the exotic structures with the same shift functor. In our case, we show that $\proj{P(\L{2})}$ cannot have an algebraic or topological triangulated structure with inverse shift functor $P\mapsto P_{\sigma}$. 

We do not know, however, whether our exotic pretriangulated structure on $\proj{P(\L{2})}$ is triangulated, i.e.~we do not know whether it satisfies Verdier's octahedral axiom. The algebra $P(\L{2})$ has finite representation type and we crucially use this to prove \Cref{thm:main}. However, this does not help with the octahedral axiom since, with our current knowledge, in order to reduce it to a finite number of computations, we would need the category of morphisms in $\modst{P(\L{2})}$ to be of finite type, but it is wild.

In \cite{cocalc-project}, we have run some tests supporting the validity of the octahedral axiom in our exotic pretriangulated structure on $\proj{P(\L{2})}$. However, therein, we also check that the sufficient condition in \cite[Theorems 6.1 and 7.7]{muro_2020_first_obstructions_enhancing} does not hold. That condition, which has a low level in the coherence hierarchy, can effectively be checked because $P(\L{2})$ has finite type.

In either case, our new exotic pretriangulated category yields the first example of one of the following phenomena:
\begin{itemize}
	\item Exotic triangulated category over a field of positive characteristic.
	\item Pretriangulated category in characteristic other than $2$ which is not triangulated.
\end{itemize}
We may even have both if the octahedral axiom depends on the characteristic.

\subsection*{Acknowledgements}

Authors were partially supported by the project PID2024-157173NB-I00 funded by MCIN/AEI/10.13039/501100011033 and by FEDER, UE. The first author acknowledges the Spanish Ministry of Science, Innovation and Universities for the financial support through the FPU predoctoral grant FPU23/01226.

Both authors are very grateful to Gustavo Jasso (Cologne) for stimulating conversations on the contents of this paper. The first author is indebted to Gustavo Jasso for his hospitality during his three-month-long visit to the University of Cologne.

\section{Triangulated categories of projectives}

In this section we work over a ground field $k$. Let $\Lambda$ be a finite\-/dimensional basic algebra. Denote by $\mod{\Lambda}$ the category of finite\-/dimensional right $\Lambda$-modules and $\proj{\Lambda}\subset\mod{\Lambda}$ the full subcategory of projectives.

Recall that a \emph{pretriangulated category} is defined in the same way as a triangulated category, but without the octahedral axiom, see \cite[\S2]{krause_2007_derived_categories_resolutions}. If $\proj{\Lambda}$ is pretriangulated then $\Lambda$ is self\-/injective as a consequence of \cite{freyd_1966_stable_homotopy}. Since this is the case we are interested in, we assume the self\-/injectivity of $\Lambda$ straight away.

The \emph{stable module category}, i.e.~the quotient of $\mod{\Lambda}$ by the ideal of morphisms factoring through a projective, is denoted by $\modst{\Lambda}$. This category is triangulated with inverse shift functor $\Omega$, the \emph{syzygy} functor.

Given an algebra automorphism $\sigma\in\Aut{\Lambda}$ and $M$ in $\mod{\Lambda}$, the restriction of scalars of $M$ along $\sigma$ is the \emph{twisted $\Lambda$-module} $M_{\sigma}$ with the same underlying vector space as $M$ and multiplication by elements of $\Lambda$, denoted  $\star$, given by
\[m\star\lambda=m\sigma(\lambda),\quad m\in M,\quad \lambda\in\Lambda.\]
Restriction of scalars acts trivially on morphisms, i.e.~the restriction of scalars of a morphism $f\colon M\to N$ along $\sigma$ is $f\colon M_{\sigma}\to N_{\sigma}$.

We will also consider $\Lambda$-bimodules with twisted right $\Lambda$-module structure. We say that $\Lambda$ is \emph{twisted $n$-periodic} if the $n$\textsuperscript{th} syzygy of $\Lambda$ as a bimodule over itself is $\Omega^n_{\Lambda^e}(\Lambda)\cong\Lambda_\sigma$ in $\modst{\Lambda^e}$ for some $\sigma\in\Aut{\Lambda}$.

Recall that a $k$-linear triangulated category $\T$ is \emph{algebraic} if it is equivalent to a full triangulated subcategory of the stable category of a $k$-linear Frobenius abelian category \cite{keller_2007_differential_graded_categories}.

\begin{theorem}[\cite{hanihara_2020_auslander_correspondence_triangulated,muro_2022_enhanced_finite_triangulated}]\label{thm:Hanihara-Muro}
	Assume that the ground field $k$ is perfect. Then, the following statements are equivalent:
	\begin{enumerate}
		\item $\proj{\Lambda}$ has a pretriangulated category structure.
		\item $\proj{\Lambda}$ has an algebraic triangulated category structure.
		\item $\Lambda$ is twisted $3$-periodic.
	\end{enumerate}
\end{theorem}

The implications $(3)\Rightarrow(2)\Rightarrow(1)$ hold without any assumptions on $k$. The second one is obvious and the first one follows from \cite[Theorem 8.1]{amiot_2007_structure_triangulated_categories}.

We emphasize the fact that being a (pre)triangulated category is a structure rather than a property. A given category may have several (pre)triangulated structures, even with the same shift functor, compare \cite{balmer_2002_triangulated_categories_several}. Moreover, \Cref{thm:Hanihara-Muro} does not assert that any pretriangulated structure on $\proj{\Lambda}$ is an algebraic triangulated structure. It just says that if it has a pretriangulated category structure then it also has a possibly different algebraic triangulated category structure. As we will see below, they may even have a different shift functor.

Although we usually designate a (pre)triangulated category by its underlying category $\T$, it really is a triple $(\T,\Omega,\triangle)$ consisting of an additive category $\T$, a self\-/equivalence $\Omega\colon\T\stackrel{\sim}{\to}\T$ called \emph{inverse shift functor}, and a class $\triangle$ of \emph{exact triangles}, which look like
\[\Omega X\stackrel{f}{\longrightarrow} Y \stackrel{i}{\longrightarrow}Z\stackrel{q}{\longrightarrow}X.\]
Similarly, an \emph{exact functor} $(\Phi,\varphi)\colon(\T,\Omega,\triangle)\to(\T',\Omega',\triangle')$ consists of a functor $\Phi\colon\T\to\T'$ together with a natural isomorphism $\varphi\colon\Omega'\Phi\cong\Phi\Omega$ such that $(\Phi,\varphi)$ takes exact triangles in $\triangle$ to exact triangles in $\triangle'$. A \emph{morphism of exact functors} $\alpha\colon(\Phi,\varphi)\to(\Psi,\psi)$ is a natural transformation $\alpha\colon\Phi\to\Psi$ such that the following square of natural transformations commutes
\begin{equation}\label{eq:triangulated_natural_transformation}
\begin{tikzcd}
	\Omega'\Phi\ar[d,"\cong","\varphi"']\ar[r,"\Omega'\alpha"]&\Omega'\Psi \ar[d,"\cong"',"\psi"]\\
	\Phi\Omega\ar[r,"\alpha\Omega"]&\Psi\Omega
\end{tikzcd}
\end{equation}

For any pretriangulated category $\T$, $(\Omega,-\id{\Omega^2})$ is an exact self\-/equivalence, and so are its powers $(\Omega^n,(-1)^n\id{\Omega^{n+1}})$.

The \emph{Picard group} $\Pic{\Lambda}$ is the group of natural isomorphism classes of self\-/equivalences of $\proj{\Lambda}$, or equivalently of $\mod{\Lambda}$. It is isomorphic to the \emph{outer automorphism group} $\Out{\Lambda}$ of the algebra $\Lambda$, see \cite[Proposition 3.8]{bolla_1984_isomorphisms_endomorphism_rings}. The isomorphism $\Out{\Lambda}\cong\Pic{\Lambda}$ takes the class of an automorphism $\sigma$ to the class of the restriction of scalars along $\sigma$, $M\mapsto M_{\sigma}$. Therefore, these are the possible shift functors of $\proj{\Lambda}$. Since $M\mapsto M_{\sigma}$ is also a self\-/equivalence of $\mod{\Lambda}$, it induces an exact self\-/equivalence of $\modst{\Lambda}$, with structure natural isomorphism $\varphi_\sigma\colon\Omega(M_\sigma)\cong(\Omega M)_\sigma$. A representative of this natural isomorphism is defined as follows. First of all, we must choose a projective cover $p_M\colon P_M\twoheadrightarrow M$ for each object $M$ of $\mod{\Lambda}$. Its kernel $\Omega M$ is the syzygy of $M$, so we have short exact sequences
\begin{equation}\label{eq:syzygy_ses}
\begin{tikzcd}
\Omega M \arrow[r, hook, "j_M"] & P_M \arrow[r, two heads, "p_M"] & M
\end{tikzcd}
\end{equation}
where $j_M$ is the inclusion of the kernel of $p_M$. We now complete the identity in $M_\sigma$ to a commutative diagram:
\[
	\begin{tikzcd}
	\Omega (M_\sigma) \arrow[r, "j_{M_\sigma}", hook] \arrow[d, "\varphi_\sigma(M)"'] & P_{M_\sigma} \arrow[r, "p_{M_\sigma}", two heads] \arrow[d] & M_\sigma \arrow[d, "\id{M_{\sigma}}"] \\
	(\Omega M)_\sigma \arrow[r, "j_M"', hook] & (P_{M})_\sigma \arrow[r, "p_{M}"', two heads] & M_\sigma
	\end{tikzcd}
\]
The class of the map $\varphi_\sigma(M)$ is well defined in $\modst{\Lambda}$. It is the natural isomorphism $\varphi_\sigma$ at $M$.

\begin{theorem}[{\cite[Theorem 16.4]{heller_1968_stable_homotopy_categories}}]\label{thm:Heller}
	Let $k$ be any field and $\sigma\in\Aut{\Lambda}$. The set of pretriangulated structures on $\proj{\Lambda}$ with inverse shift functor $M\mapsto M_{\sigma}$ is in bijection with the set of isomorphisms of exact functors $(\Omega^{3},-\id{\Omega^{4}})\cong((-)_\sigma,\varphi_\sigma)$ between these exact self\-/equivalences of $\modst{\Lambda}$.
\end{theorem}

\Cref{thm:Hanihara-Muro} characterizes when this set is non-empty. Many of the \emph{different} pretriangulated structures in \Cref{thm:Heller} may be \emph{equivalent}, though.

\Cref{thm:Heller} is originally stated in terms of shift functors, not inverse shift functors. Both approaches are obviously equivalent since they are inverse self-equivalences.

\begin{remark}\label{rem:Heller}
	Given an isomorphism of exact functors $\delta\colon(\Omega^{3},-\id{\Omega^{4}})\cong((-)_\sigma,\varphi_\sigma)$, the exact triangles of the associated pretriangulated structure on $\proj{\Lambda}$ can be described as follows. An exact triangle must be a sequence in $\proj{\Lambda}$ of the form
	\[P_{\sigma} \stackrel{f}\longrightarrow Q\stackrel{i}\longrightarrow R \stackrel{q}\longrightarrow P\]
	such that
	\[R_{\sigma}\stackrel{q}{\longrightarrow}P_{\sigma} \stackrel{f}\longrightarrow Q\stackrel{i}\longrightarrow R \stackrel{q}\longrightarrow P\]
	is exact. Let $M=\coker q$ in $\mod{\Lambda}$. The previous sequence identifies
	\[\Omega^{3}M\cong\ker \left(i\colon Q\to R\right) \cong\coker\left(q\colon R_{\sigma}\to P_{\sigma}\right)=M_{\sigma}\]
	in $\modst{\Lambda}$. The sequence is an exact triangle if and only if the previous isomorphism coincides with $\delta(M)$.

	As a consequence of the previous description, exact triangles in the pretriangulated category $(\proj{\Lambda},(-)_\sigma,\triangle_{\delta})$ decompose as a direct sum of trivial triangles and \emph{minimal exact triangles}. The latter are those obtained from a non-projective indecomposable $M$ in $\mod{\Lambda}$ by taking a representative $\Omega^{3}M\cong M_\sigma$ of $\delta(M)$ and the first steps of a minimal projective resolution
	\[
		\begin{tikzcd}
			P_{\sigma}\ar[rr,"f"]\ar[rd,"p"']&&Q\ar[r,"i"]&R\ar[r,"q"]&P\ar[rd,two heads,"p"]&\\
			&M_\sigma\cong \Omega^{3}M\ar[ru, hook, "j"']&&&&M
		\end{tikzcd}
	\]
\end{remark}

We define the \emph{stable Picard group} $\Picst{\Lambda}$ as the group of isomorphism classes of exact self-equivalences of $\modst{\Lambda}$. Different groups in the literature go by this name, e.g.~\citeauthor{asashiba_2003_lift_individual_stable} considers the group of self-equivalences of $\modst{\Lambda}$ as an ordinary category \cite{asashiba_2003_lift_individual_stable}, see \cite{dugas_trok_2018_stable_picard_group} for yet another candidate. This is because different but equivalent definitions for abelian categories yield possibly different triangulated versions. \citeauthor{asashiba_2003_lift_individual_stable}'s stable Picard group sits below our $\Picst{\Lambda}$. There is an obvious group morphism $\Pic{\Lambda}\to\Picst{\Lambda}$ sending the class of a self\-/equivalence of $\mod{\Lambda}$ to the class of the induced self\-/equivalence of $\modst{\Lambda}$.

The algebra $\Lambda$ is \emph{connected} or \emph{indecomposable} if it is not isomorphic to the product of two non-trivial algebras.

\begin{theorem}\label{thm:non-algebraic}
	Let $\Lambda$ be a basic connected non\-/separable self\-/injective twisted $3$-periodic finite\-/dimensional algebra over a perfect field $k$. The category $\proj{\Lambda}$ has an essentially unique algebraic triangulated structure. Its inverse shift functor is $M\mapsto M_{\nu}$, the restriction of scalars along an automorphism $\nu\in\Aut{\Lambda}$ with $\Omega^{3}_{\Lambda^e}(\Lambda)\cong\Lambda_{\nu}$ in $\modst{\Lambda^e}$. Moreover, any non-trivial class $[\phi]$ in the kernel of the group morphism $\Out{\Lambda}\cong\Pic{\Lambda}\to\Picst{\Lambda}$ gives rise to a non-algebraic pretriangulated category structure on $\proj{\Lambda}$ with inverse shift functor $M\mapsto M_{\sigma}$, the restriction of scalars along $\sigma=\nu\phi$.
\end{theorem}

\begin{proof}
	The claim about the existence of an essentially unique algebraic triangulated category structure on $\proj{\Lambda}$, and about its inverse shift functor, follows from \cite[Proposition 9.8]{muro_2022_enhanced_finite_triangulated}.

	Let
	\[\delta\colon (\Omega^{3},-\id{\Omega^{4}})\cong((-)_{\nu},\varphi_{\nu})\]
	be the isomorphism of exact functors corresponding to the algebraic triangulated category structure on $\proj{\Lambda}$ through \Cref{thm:Heller}.

	Since $[\phi]$ is in the kernel of $\Out{\Lambda}\cong\Pic{\Lambda}\to\Picst{\Lambda}$, we have another isomorphism of exact functors
	\[\zeta\colon(\id{\modst{\Lambda}},\id{\Omega})\cong((-)_{\phi},\varphi_{\phi}).\]

	We can horizontally compose both isomorphisms of exact functors, obtaining a new one
	\[\zeta\delta\colon(\Omega^{3},-\id{\Omega^{4}})\cong((-)_{\nu\phi},\varphi_{\nu\phi})=((-)_{\sigma},\varphi_{\sigma}).\]
	This isomorphism of exact functors induces a pretriangulated category structure on $\proj{\Lambda}$ with inverse shift functor $M\mapsto M_{\sigma}$ through \Cref{thm:Heller}.

	The natural equivalences $M\mapsto M_{\nu}$ and $M\mapsto M_{\sigma}$ of $\proj{\Lambda}$ are not naturally isomorphic because $[\nu]\neq[\sigma]=[\nu][\phi]\in\Out{\Lambda}\cong\Pic{\Lambda}$ since $[\phi]$ is non-trivial. Therefore, the latter pretriangulated structure on $\proj{\Lambda}$ cannot be equivalent to its essentially unique algebraic triangulated structure.
\end{proof}

\begin{remark}\label{rem:exact-triangles}
	Exact triangles in a non-algebraic pretriangulated structure on $\proj{\Lambda}$ obtained from \Cref{thm:non-algebraic} can be obtained from those of an algebraic triangulated structure in the following way. Fix the latter.

	Assume we have a minimal exact triangle in the fixed algebraic triangulated category structure on $\proj{\Lambda}$,
	\[P_\nu\stackrel{f}{\longrightarrow}Q\stackrel{i}{\longrightarrow}R\stackrel{q}{\longrightarrow}P.\]
	Let $M=\coker q$, which is a non-projective indecomposable in $\mod{\Lambda}$, and factor $f$ as follows:
	\[
	\begin{tikzcd}
		P_{\nu}\ar[rr,"f"]\ar[rd,two heads,"p"']&&Q\ar[r,"i"]&R\ar[r,"q"]&P\ar[rd,two heads, "p"]&\\
		&M_\nu\ar[ru,"j"', hook]&&&&M
	\end{tikzcd}
	\]

	Choose a representative $M_{\nu}\cong M_{\nu\phi}=M_{\sigma}$ in $\mod{\Lambda}$ of the isomorphism $\zeta(M_{\nu})$ in $\modst{\Lambda}$ in the proof of \Cref{thm:non-algebraic}. We can indeed choose the representative to be an isomorphism. Since $\zeta$ is natural, it suffices to prove it for $M$ indecomposable. If $M$ is non-projective, any representative is actually an isomorphism. If $M=e\Lambda$, with $e$ belonging to a complete orthogonal set of primitive idempotents, then we can take $e\Lambda_{\sigma}\cong e\Lambda_{\nu}$ to be the only isomorphism sending $e\mapsto e$.

	The following sequence is a minimal exact triangle in the non-algebraic pretriangulated category structure on $\proj{\Lambda}$:
	\[
	\begin{tikzcd}
		P_{\sigma}\ar[rr,"f"]\ar[rd,two heads,"p"']&&Q\ar[r,"i"]&R\ar[r,"q"]&P\\
		&M_{\sigma}\cong M_\nu\ar[ru,"j"', hook]&&&
	\end{tikzcd}
	\]
\end{remark}

Obtaining examples of algebras satisfying the hypotheses of \Cref{thm:non-algebraic} and such that the group morphism $\Out{\Lambda}\cong\Pic{\Lambda}\to\Picst{\Lambda}$ is not injective is not easy. Our example is computed in \Cref{sec:example}. Some readers may want to look for other examples. In order to help them, we present a criterion to discard some non-examples.

The algebra $\Lambda$ is \emph{Schurian} if $\dim(e_i\Lambda e_j)\leq 1$ for any pair $e_i,e_j$ belonging to a complete orthogonal set of primitive idempotents.

\begin{lemma}[{\cite[Lemma 4.4]{dugas_2010_periodic_resolutions_selfinjective}}]\label{lem:Dugas}
	If $\Lambda$ is Schurian then the group morphism $\Pic{\Lambda}\to\Picst{\Lambda}$ is injective.
\end{lemma}

Actually, \citeauthor{dugas_2010_periodic_resolutions_selfinjective} proves this \namecref{lem:Dugas} for \citeauthor{asashiba_2003_lift_individual_stable}'s stable Picard group, hence it follows for ours, which sits above.

\begin{example}
	The following example shows that the condition of being Schurian in \Cref{lem:Dugas} is not necessary.

	Let $\Lambda=k[\varepsilon]/(\varepsilon^2)$ be the algebra of dual numbers over a perfect field $k$, which is obviously not Schurian. The algebra automorphism group is $k^\times\cong\Aut{\Lambda}$, sending $\alpha\in k^\times$ to the automorphism $\sigma_\alpha\in\Aut{\Lambda}$ defined by $\sigma_\alpha(\varepsilon)=\alpha\varepsilon$, and $\Pic{\Lambda}\cong\Out{\Lambda}=\Aut{\Lambda}$ since $\Lambda$ is commutative. This algebra satisfies all hypotheses of \Cref{thm:non-algebraic} with $\Omega^3_{\Lambda^e}(\Lambda)\cong\Lambda_{\sigma_{-1}}$ in \(\modst{\Lambda^e}\).

	The stable module category $\modst{\Lambda}$ is equivalent to $\mod{k}$. The equivalence is the composition $\mod{k}\to\mod{\Lambda}\twoheadrightarrow\modst{\Lambda}$ of the full inclusion defined by the algebra epimorphism $p\colon \Lambda\twoheadrightarrow k$ followed by the canonical projection onto the quotient category.
	\citeauthor{asashiba_2003_lift_individual_stable}'s stable Picard group is therefore trivial. However, our $\Picst{\Lambda}$ is non-trivial in general. Moreover, the group morphism $\Pic{\Lambda}\to\Picst{\Lambda}$ is an isomorphism.

	Indeed, the induced triangulated structure on $\mod{k}\simeq\modst{\Lambda}$ is the only possible one, with shift functor the identity and exact triangles the direct sums of the trivial ones. Moreover, the ordinary self\-/equivalence of $\mod{k}$ induced by any automorphism $\sigma_\alpha$ of $\Lambda$ is the identity. However, the \emph{exact} self-equivalence is not, since the natural isomorphism $\varphi_{\sigma_\alpha}\colon k\cong k$ is multiplication by $\alpha$.
	This holds because we have a commutative diagram in $\mod{\Lambda}$ as follows:
	\[
	\begin{tikzcd}
	k \arrow[r, "i", hook] \arrow[d, "\alpha"'] & \Lambda \arrow[r, "p", two heads] \arrow[d, "f"'] & k \arrow[d, "\id{k}"] \\
	k \arrow[r, "i"', hook] & \Lambda_{\sigma_\alpha} \arrow[r, "p"', two heads] & k
	\end{tikzcd}
	\]
	Here $i(1)=\varepsilon$ and $f\colon\Lambda\to\Lambda_{\sigma_\alpha}$ is the only morphism in $\mod{\Lambda}$ satisfying $f(1)=1$. The square on the right is obviously commutative, and the square on the left commutes since
	\[
		fi(1)=f(\varepsilon)=f(1)\star \varepsilon=\sigma_{\alpha}(\varepsilon)=\alpha\varepsilon=i(\alpha).
	\]
\end{example}

\section{The example}\label{sec:example}

In this section we consider the preprojective algebra of generalized Dynkin type $\L{2}$ over an algebraically closed field $k$,
\[\Lambda=P(\L{2}),\]
presented in \eqref{eq:presentation}. A basis of this algebra is given by
\begin{equation}\label{eq:basis}
	\begin{array}{rlrl}
		\{e_1,b,b^2,b^3\}&\subset e_1\Lambda e_1,\hspace{2cm}&
		\{a_1,ba_1\}&\subset e_1\Lambda e_2,\\
		\{a_2,a_2b\}&\subset e_2\Lambda e_1,&
		\{e_2,a_2ba_1\}&\subset e_2\Lambda e_2.
	\end{array}
\end{equation}
Here $\{e_1, e_2\}$ is the complete orthogonal set of primitive idempotents corresponding to the vertices. The \emph{radical} of $\Lambda$ will be denoted by $\rad{\Lambda}\subset\Lambda$. A basis of $\rad{\Lambda}$ is given by the elements of the previous bases except for $e_1$ and $e_2$.

The algebra $\Lambda$ satisfies all assumptions in \Cref{thm:non-algebraic}, in particular it is self\-/injective (actually symmetric by \cite[Corollary 4]{bialkowski_erdmann_skowronski_2011_deformed_preprojective_algebras}) and $\Lambda$ is twisted $3$-periodic with $\Omega^{3}_{\Lambda^e}\Lambda\cong \Lambda_{\nu}$ in \(\modst{\Lambda^e}\) for the automorphism $\nu\in\Aut{\Lambda}$ defined by
\begin{equation}\label{eq:nu_automorphism}
	\nu\colon \Lambda \stackrel{\cong}{\longrightarrow} \Lambda\colon\;\;
	e_1 \mapsto e_1, \;\;
	e_2 \mapsto e_2, \;\;
	a_1 \mapsto -a_1, \;\;
	a_2 \mapsto -a_2, \;\;
	b \mapsto -b.
\end{equation}
An explicit exact sequence of $\Lambda$-bimodules
\begin{equation*}\label{eq:extension}
	\Lambda_{\nu}\hookrightarrow Q_2\longrightarrow Q_1\longrightarrow Q_0\twoheadrightarrow \Lambda
\end{equation*}
with projective middle terms
\begin{align}
	\label{eq:cover_bimodule} Q_0=Q_2&=\Lambda e_1\otimes e_1\Lambda\oplus \Lambda e_2\otimes e_2\Lambda,\\
	\nonumber Q_1&=\Lambda e_1\otimes e_2\Lambda\oplus \Lambda e_2\otimes e_1\Lambda\oplus \Lambda e_1\otimes e_1\Lambda,
\end{align}
is constructed in \cite[Proposition 2.3]{bialkowski_erdmann_skowronski_2007_deformed_preprojective_algebras}. Here, each $Q_n$ is the projective cover of the $n$\textsuperscript{th} syzygy of $\Lambda$.

\begin{proposition}\label{prop:exotic_automorphism}
	For any algebraically closed field $k$, the algebra automorphism $\phi\in\Aut{\Lambda}$ defined by
	\[\phi\colon \Lambda \stackrel{\cong}{\longrightarrow} \Lambda\colon\;\;
	e_1 \mapsto e_1, \;\;
	e_2 \mapsto e_2, \;\;
	a_1 \mapsto a_1, \;\;
	a_2 \mapsto a_2, \;\;
	b \mapsto b+b^3,\]
	represents a non-trivial class in $\Out{\Lambda}$ which is in the kernel of $\Out{\Lambda}\cong\Pic{\Lambda}\to\Picst{\Lambda}$.
\end{proposition}

\begin{proof}
	It is straightforward to check that the formulas in the statement are compatible with the relations, so they define an endomorphism $\phi$ of $\Lambda$. It is an isomorphism because it fixes all elements in the basis of $\Lambda$ in \eqref{eq:basis} except for $b$, which is sent to $\phi(b)=b+b^3$. Therefore, the matrix of $\phi$ with respect to that basis is an elementary matrix.

	If $\phi$ were an inner automorphism, there would exist a unit $u\in\Lambda^{\times}$ such that
	\[ub=(b+b^3)u.\]
	Assume this happens.
	Then, in particular, $e_1ub=(b+b^3)ue_1$.
	Since $e_1\Lambda e_1=k[b]/(b^4)$ is commutative,
	\begin{equation*}
		e_1ub=e_1ue_1b
		=be_1ue_1
		=bue_1,
	\end{equation*}
	therefore $b^3ue_1=0$. This is impossible because $\Lambda/\rad{\Lambda}\cong k\times k$ so $u$ modulo the radical is $\alpha_1e_1+\alpha_2e_2$ for some $\alpha_1,\alpha_2\in k$, $\alpha_1\neq0\neq\alpha_2$. Since $\rad{\Lambda}^4=0$, $b^3u=b^3(\alpha_1e_1+\alpha_2e_2)$, hence	$b^3ue_1=\alpha_1b^3\neq0$.

	We finally check that $[\phi]\in\Out{\Lambda}$ is in the kernel of $\Out{\Lambda}\cong\Pic{\Lambda}\to\Picst{\Lambda}$. We must define an isomorphism of exact functors
	\[\alpha(M)\colon M\stackrel{\cong}{\longrightarrow} M_{\phi},\qquad \alpha\colon(\id{\modst{\Lambda}},\id{\Omega})\stackrel{\cong}{\longrightarrow}((-)_{\phi},\varphi_{\phi}).\]
	By \cite[Proposition 8.69]{jensen_lenzing_1989_modeltheoretic_algebra_particular}, $\soc{\Lambda} = (b^3,a_2ba_1)$ annihilates any indecomposable non-projective object $M$ in $\mod{\Lambda}$, in particular $M_\phi=M$. Any object in $\modst{\Lambda}$ is isomorphic to a direct sum of indecomposable non-projective objects in $\mod{\Lambda}$. For any such object we can set
	\[\alpha(M)=\id{M}.\]
	This defines a natural isomorphism
	\[\alpha(M)\colon M\stackrel{\cong}{\longrightarrow} M_{\phi},\qquad \alpha\colon\id{\modst{\Lambda}}\stackrel{\cong}{\longrightarrow}(-)_{\phi}.\]
	It is only left to check that the square \eqref{eq:triangulated_natural_transformation} commutes in this case, i.e.~the following square in $\modst{\Lambda}$ commutes:
	\begin{equation}\label{eq:commutative_square}
		\begin{tikzcd}[column sep=large]
			\Omega(M)\ar[d,"\cong","\id{\Omega(M)}"']\ar[r,"\Omega\alpha(M)"]&\Omega(M_{\phi}) \ar[d,"\cong"',"\varphi_{\phi}(M)"]\\
			\Omega(M)\ar[r,"\alpha(\Omega(M))"]&\Omega(M)_{\phi}
		\end{tikzcd}
	\end{equation}
	Since $\Omega(M)$ is also a direct sum of indecomposable non-projective objects in $\mod{\Lambda}$, we can omit the twisting by $\phi$ in all modules in \eqref{eq:commutative_square}. By definition of $\alpha$, horizontal maps in \eqref{eq:commutative_square} are identity maps, so the only possible non-identity map is $\varphi_{\phi}(M)$, which fits in a commutative diagram as follows, where $p\colon P\twoheadrightarrow M$ is a chosen projective cover:
	\[
	\begin{tikzcd}
	\Omega M \arrow[r, "j", hook] \arrow[d, "\varphi_{\phi}(M)"'] & P \arrow[r, "p", two heads] \arrow[d, "f"'] & M \arrow[d, "\id{M}"] \\
	\Omega M \arrow[r, "j"', hook] & P_{\phi} \arrow[r, "p"', two heads] & M
	\end{tikzcd}
	\]
	In order to show that \eqref{eq:commutative_square} is commutative, we must prove that $\varphi_{\phi}(M)$ is the identity in $\modst{\Lambda}$.
	This is equivalent to showing that we can find a map $f\colon P\to P_{\phi}$ such that the following diagram commutes in $\mod{\Lambda}$:
	\begin{equation}\label{eq:trivialization}
	\begin{tikzcd}
	\Omega M \arrow[r, "j", hook] \arrow[d, "\id{\Omega M}"'] & P \arrow[r, "p", two heads] \arrow[d, "f"'] & M \arrow[d, "\id{M}"] \\
	\Omega M \arrow[r, "j"', hook] & P_{\phi} \arrow[r, "p"', two heads] & M
	\end{tikzcd}
	\end{equation}
	We now explain how to obtain $f$ in the following two cases: if $(\Omega M)\cdot b^2=0$ and if $M\cdot b^2=0$.

	The indecomposable projective $e_2\Lambda$ satisfies $e_2\Lambda = e_2\Lambda_\phi$ because $e_2\Lambda b^3=0$, actually $e_2\Lambda b^2=0$. For $e_1\Lambda$, the unique morphism $\zeta\colon e_1\Lambda \to e_1\Lambda_\phi$ in $\mod{\Lambda}$ fixing $e_1$ is an isomorphism given by $\zeta(e_1\lambda)= e_1\phi(\lambda)$. Let $f\colon P\to P_{\phi}$ be the isomorphism defined by a choice of direct sum decomposition $P\cong e_1\Lambda^p \oplus e_2\Lambda^q$, the isomorphism $\zeta$ and the identity in $e_2\Lambda$. The square on the right of \eqref{eq:trivialization} commutes because $M\cdot b^3=0$. If $(\Omega M)\cdot b^2=0$ then the square on the left also commutes. Indeed, $\Omega M\subset \rad{P}$, $\phi$ fixes all elements in the following basis of $\rad{e_1\Lambda}$ except for $b$:
	\begin{equation}\label{eq:basis_radical_e1}
			\{b,b^2,b^3,a_1,ba_1\}\subset \rad{e_1\Lambda},
	\end{equation}
	and any $\lambda\in\rad{e_1\Lambda}$ with $\lambda\cdot b^2=0$ must have trivial coefficient in $b$ since $x\cdot b^2=0$ for any element in the previous basis other than $x=b$ and $b\cdot b^2=b^3\neq 0$. Therefore, $f$ fixes the subspace of $\rad{P}$ annihilated by $b^2$, in particular it restricts to the identity in $\Omega M$.

	The unique morphism $\xi\colon e_1\Lambda \to e_1\Lambda_\phi$ in $\mod{\Lambda}$ mapping $\xi(e_1)= e_1-b^2$ is an isomorphism with $\xi(e_1\lambda)= (e_1-b^2)\phi(\lambda)$. As above, this isomorphism gives rise to another $f\colon P\to P_{\phi}$. If $M\cdot b^2=0$ then the square on the right of \eqref{eq:trivialization} commutes. The square on the left also commutes, even without that hypothesis. Indeed, $\xi$ restricts to the identity in $\rad{e_1\Lambda}$ because, if $x\neq b$ belongs to the basis \eqref{eq:basis_radical_e1} then $b^2\cdot x=0$ so $\xi(x)=x$, and $\xi(b)=(e_1-b^2)(b+b^3)=b$ too. Hence, the new $f$ restricts to the identity in $\rad{P}\supset\Omega M$.

	It suffices to construct an $f$ making \eqref{eq:trivialization} commutative for any indecomposable non-projective $M$ in $\mod{\Lambda}$. We now describe these objects and check that all of them, except for two, satisfy $(\Omega M)\cdot b^2=0$ or $M\cdot b^2=0$, so they are covered by the two previous paragraphs. The two exceptional cases will be treated separately.

	The algebra $\Lambda$ has finite representation type. Its stable Auslander--Reiten quiver was computed in \cite{crawley-boevey_2019_noncommutative_algebra_1}, by endowing $\Lambda/\soc{\Lambda}$ with the grading where all arrows have degree $1$:
	\[
	\begin{tikzcd}
	1 \arrow[r, shift left, "a_1"] \arrow[loop right, "b"', in=210,out=150,looseness=5] & 2, \arrow[l, shift left, "a_2"]
	\end{tikzcd}
	\qquad
	a_1a_2=b^2,\qquad a_2a_1=0,\qquad b^3=0,\qquad a_2ba_1=0,
	\]
	and considering the covering
	
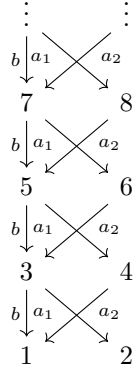
\[
	\begin{tikzcd}
		\vdots\ar[d,"b"']\ar[rd,"a_1"', near start]&\vdots\ar[ld,"a_2", near start]\\
		7\ar[d,"b"']\ar[rd,"a_1"', near start]&8\ar[ld,"a_2", near start]\\
		5\ar[d,"b"']\ar[rd,"a_1"', near start]&6\ar[ld,"a_2", near start]\\
		3\ar[d,"b"']\ar[rd,"a_1"', near start]&4\ar[ld,"a_2", near start]\\
		1&2
	\end{tikzcd}
	\]
	The following piece of the Auslander--Reiten quiver of this covering projects to the stable Auslander--Reiten quiver of $\Lambda$, identifying the left and right boundaries as indicated:
	
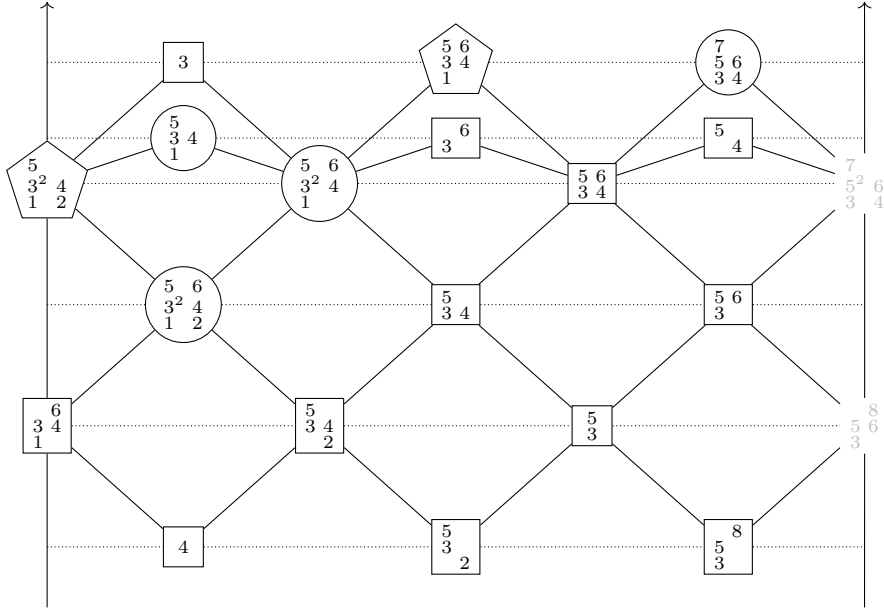
\begin{equation*}
		\begin{tikzpicture}[xscale = 1.8, yscale = .8]
			\node [draw, inner sep=2pt, minimum size=1.5em] (3) at (0,0) {$\begin{smallmatrix*}[l]3\end{smallmatrix*}$};
			\node [draw, regular polygon,regular polygon sides=5, inner sep=0pt, minimum size=1.5em] (53164) at (2,0) {$\begin{smallmatrix*}[l]5&6\\3&4\\1&\end{smallmatrix*}$};
			\node [draw, circle, inner sep=1pt, minimum size=1.5em] (75364) at (4,0) {$\begin{smallmatrix*}[l]7&\\5&6\\3&4\end{smallmatrix*}$};
			\node [draw, circle, inner sep=1pt, minimum size=1.5em] (5314) at (0,-1.25) {$\begin{smallmatrix*}[l]5&\\3&4\\1&\end{smallmatrix*}$};
			\node [draw, inner sep=2pt, minimum size=1.5em] (36) at (2,-1.25) {$\begin{smallmatrix*}[l]&6\\3&\end{smallmatrix*}$};
			\node [draw, inner sep=2pt, minimum size=1.5em] (54) at (4,-1.25) {$\begin{smallmatrix*}[l]5&\\&4\end{smallmatrix*}$};
			\node [draw, regular polygon,regular polygon sides=5, inner sep=0pt, minimum size=1.5em] (532142) at (-1,-2) {$\begin{smallmatrix*}[l]5&\\3^2&4\\1&2\end{smallmatrix*}$};
			\node [draw, circle, inner sep=1pt, minimum size=1.5em] (532164) at (1,-2) {$\begin{smallmatrix*}[l]5&6\\3^2&4\\1&\end{smallmatrix*}$};
			\node [draw, inner sep=2pt, minimum size=1.5em] (5364) at (3,-2) {$\begin{smallmatrix*}[l]5&6\\3&4\end{smallmatrix*}$};
			\node [inner sep=2pt, minimum size=1.5em, lightgray] (752364) at (5,-2) {$\begin{smallmatrix*}[l]7&\\5^2&6\\3&4\end{smallmatrix*}$};
			\node [draw, circle, inner sep=1pt, minimum size=1.5em] (5321642) at (0,-4) {$\begin{smallmatrix*}[l]5&6\\3^2&4\\1&2\end{smallmatrix*}$};
			\node [draw, inner sep=2pt, minimum size=1.5em] (534) at (2,-4) {$\begin{smallmatrix*}[l]5&\\3&4\end{smallmatrix*}$};
			\node [draw, inner sep=2pt, minimum size=1.5em] (536) at (4,-4) {$\begin{smallmatrix*}[l]5&6\\3&\end{smallmatrix*}$};
			\node [draw, inner sep=2pt, minimum size=1.5em] (3164) at (-1,-6) {$\begin{smallmatrix*}[l]&6\\3&4\\1&\end{smallmatrix*}$};
			\node [draw, inner sep=2pt, minimum size=1.5em] (5342) at (1,-6) {$\begin{smallmatrix*}[l]5&\\3&4\\&2\end{smallmatrix*}$};
			\node [draw, inner sep=2pt, minimum size=1.5em] (53) at (3,-6) {$\begin{smallmatrix*}[l]5\\3\end{smallmatrix*}$};
			\node [inner sep=2pt, minimum size=1.5em, lightgray] (5386) at (5,-6) {$\begin{smallmatrix*}[l]&8\\5&6\\3&\end{smallmatrix*}$};
			\node [draw, inner sep=2pt, minimum size=1.5em] (4) at (0,-8) {$\begin{smallmatrix*}[l]4\end{smallmatrix*}$};
			\node [draw, inner sep=2pt, minimum size=1.5em] (532) at (2,-8) {$\begin{smallmatrix*}[l]5&\\3&\\&2\end{smallmatrix*}$};
			\node [draw, inner sep=2pt, minimum size=1.5em] (538) at (4,-8) {$\begin{smallmatrix*}[l]&8\\5&\\3&\end{smallmatrix*}$};
			\draw[densely dotted] (-1,0) to (3) to (53164) to (75364) to (5,0);
			\draw[densely dotted] (-1,-1.25) to (5314) to (36) to (54) to (5,-1.25);
			\draw[densely dotted] (532142) to (532164) to (5364) to (752364);
			\draw[densely dotted] (-1,-4) to (5321642) to (534) to (536) to (5,-4);
			\draw[densely dotted] (3164) to (5342) to (53) to (5386);
			\draw[densely dotted] (-1,-8) to (4) to (532) to (538) to (5,-8);
			\draw[->] (-1,-9) to (3164) to (532142) to (-1,1);
			\draw[->] (5,-9) to (5386) to (752364) to (5,1);
			\draw[-] (532142) to (3) to (532164) to (53164) to (5364) to (75364) to (752364) to (54) to (5364) to (36) to (532164) to (5314) to (532142) to (5321642) to (532164) to (534) to (5364) to (536) to (752364);
			\draw[-] (3164) to (5321642) to (5342) to (534) to (53) to (536) to (5386) to (538) to (53) to (532) to (5342) to (4) to (3164);
		\end{tikzpicture}
		\end{equation*}
	Here, dimension vectors are indicated by placing each vertex's dimension as an exponent (omitting exponents equal to $1$ and $0$-dimensional vertices). Since $\Lambda$ is symmetric, the Auslander--Reiten translation is $\tau=\Omega^2$ (\cite[Corollary 8.6]{skowronski_yamagata_2011_frobenius_algebras}). Its action on the stable Auslander--Reiten quiver consists of moving one step to the left along the dotted lines. Moreover, $\Omega^3$ acts as the identity by \cite[Corollary 2.5]{bialkowski_erdmann_skowronski_2007_deformed_preprojective_algebras}, although $\Omega^3$ is not naturally isomorphic to the identity. Therefore, $\Omega$ moves vertices of the stable Auslander--Reiten quiver one step to the right.

	The vertices of the stable Auslander--Reiten quiver are decorated with three kinds of shapes: squares, circles and pentagons.
	The dimension vectors show that squares satisfy $M\cdot b^2=0$, and circles satisfy $(\Omega M)\cdot b^2=0$ since the vertex on the right is a square. In order to finish this proof, we now construct a morphism $f$ making \eqref{eq:trivialization} commutative for the two pentagons.

	The pentagon in the top row is $M=\Lambda/(b-a_2)\Lambda$. In this case \eqref{eq:trivialization} is
	\begin{equation*}
		\begin{tikzcd}
		(b-a_2)\Lambda \arrow[r, "j", hook] \arrow[d, "\id{}"'] & \Lambda \arrow[r, "p", two heads] \arrow[d, "f"'] & \Lambda/(b-a_2)\Lambda \arrow[d, "\id{}"] \\
		(b-a_2)\Lambda \arrow[r, "j"', hook] & \Lambda_{\phi} \arrow[r, "p"', two heads] & \Lambda/(b-a_2)\Lambda
		\end{tikzcd}
	\end{equation*}
	where $f$ is the only morphism of $\Lambda$-modules mapping $f(1)=1-(b-a_2)b$, which is given by $f(\lambda)=(1-(b-a_2)b)\phi(\lambda)$. The square on the right commutes because $pf(\lambda)=p\phi(\lambda)=p(\lambda)$ since $p(b-a_2)=0$ and $M\cdot b^3=0$. The square on the left is also commutative since
	\begin{align*}
		fj((b-a_2)\lambda)&=(1-(b-a_2)b)\phi((b-a_2)\lambda)\\
		&=(1-b^2+a_2b)(b+b^3-a_2)\phi(\lambda)\\
		&=(b-a_2)\phi(\lambda)\\
		&=(b-a_2)\lambda\\
		&=j((b-a_2)\lambda).
	\end{align*}
	Here we use that $(b-a_2)b^3=0$.

	The other pentagon, in the third row, is
	\[M=\frac{e_1\Lambda\oplus e_1\Lambda}{(b^2,-b)\Lambda+(0,a_1)\Lambda}.\]
	In this case \eqref{eq:trivialization} is
	\begin{equation*}
		\begin{tikzcd}
		(b^2,-b)\Lambda+(0,a_1)\Lambda \arrow[r, "j", hook] \arrow[d, "\id{}"'] & e_1\Lambda\oplus e_1\Lambda \arrow[r, "p", two heads] \arrow[d, "f"'] & M \arrow[d, "\id{}"] \\
		(b^2,-b)\Lambda+(0,a_1)\Lambda \arrow[r, "j"', hook] & e_1\Lambda_{\phi}\oplus e_1\Lambda_{\phi} \arrow[r, "p"', two heads] & M
		\end{tikzcd}
	\end{equation*}
	where $f$ is the only morphism of $\Lambda$-modules satisfying
	\begin{align*}
		f(e_1,0)&=(e_1,0)-(b^2,-b),&
		f(0,e_1)&=(0,e_1),
	\end{align*}
	which is given by
	\[f(\lambda_1,\lambda_2)=(\phi(\lambda_1),\phi(\lambda_2))-(b^2,-b)\phi(\lambda_1).\]
	The square on the right commutes because $pf(\lambda_1,\lambda_2)=p(\phi(\lambda_1),\phi(\lambda_2))=p(\lambda_1,\lambda_2)$ since $p(b^2,-b)=0$ and $M\cdot b^3=0$. The square on the left is also commutative since
	\begin{align*}
		fj((b^2,-b)\lambda_1+(0,a_1)\lambda_2)&=f(b^2\lambda_1,a_1\lambda_2-b\lambda_1)\\
		&=(\phi(b^2\lambda_1),\phi(a_1\lambda_2-b\lambda_1))-(b^2,-b)\phi(b^2\lambda_1)\\
		&=(b^2\phi(\lambda_1),a_1\phi(\lambda_2)-(b+b^3)\phi(\lambda_1))-(b^2,-b)b^2\phi(\lambda_1)\\
		&=(b^2\lambda_1,a_1\lambda_2-(b+b^3)\lambda_1)-(b^2,-b)b^2\lambda_1\\
		&=(b^2\lambda_1,a_1\lambda_2-(b+b^3)\lambda_1+b^3\lambda_1)\\
		&=(b^2\lambda_1,a_1\lambda_2-b\lambda_1)\\
		&=j((b^2,-b)\lambda_1+(0,a_1)\lambda_2).
	\end{align*}
	Here we use that $b\cdot b^3=0$ and $a_1\cdot b^3=0$.
\end{proof}

\begin{corollary}\label{cor:non-algebraic}
	Let $k$ be any algebraically closed field. The category $\proj{P(\L{2})}$ has a pretriangulated category structure with inverse shift functor $M\mapsto M_\sigma$, for $\sigma=\nu\phi$ the composition of the automorphisms in \eqref{eq:nu_automorphism} and \Cref{prop:exotic_automorphism}, which is not an algebraic triangulated category structure.
\end{corollary}

The automorphism $\sigma$ in \eqref{eq:sigma_automorphism} is $\sigma=\nu\phi$ for $\nu$ in \eqref{eq:nu_automorphism} and $\phi$ in \Cref{prop:exotic_automorphism}. Hence, this \namecref{cor:non-algebraic} follows from \Cref{thm:non-algebraic,prop:exotic_automorphism}. This proves the `non-algebraic' part of \Cref{thm:main}.

\section{Not even topological}\label{sec:not_topological}

Let $k$ be a perfect field and $\Lambda$ a finite-dimensional $k$-algebra. \Cref{thm:non-algebraic} gives a sufficient condition for the existence of a pretriangulated category structure on $\proj{\Lambda}$ which is not an algebraic triangulated structure over $k$. In this section we introduce further hypotheses guaranteeing that it is not topological, and in particular it is not algebraic over any other commutative ring.

Let $k$ be a commutative ring. Recall that a $k$-linear triangulated category $\T$ is \emph{algebraic} if it is  equivalent to a full triangulated subcategory of the stable category of a $k$-linear Frobenius abelian category \cite{keller_2007_differential_graded_categories}. Recall also that $\T$ is topological if it is equivalent to a full triangulated subcategory of the homotopy category of a stable model category \cite{schwede_2010_algebraic_topological_triangulated}.

Since $\proj{\Lambda}$ is small, if it is algebraic over some commutative ring then it is equivalent to a full triangulated subcategory of the derived category $\D{\A}$ of a small DG category $\A$ \cite[Theorem 3.8]{keller_2007_differential_graded_categories}, hence it is also topological \cite{schwede_2010_algebraic_topological_triangulated}.

\begin{theorem}\label{thm:non-topological}
	We place ourselves under the assumptions of \Cref{thm:non-algebraic}. Suppose further that the following conditions are satisfied:
	\begin{enumerate}
		\item\label{it:algebra_topological} There exists a $\Z$-algebra (i.e.~a ring) $\Lambda^{\Z}$ which is free as a $\Z$-module such that $\Lambda\cong\Lambda^{\Z}\otimes_{\Z}k$.
		\item\label{it:automorphism_topological} There exists $\sigma^{\Z}\in\Aut{\Lambda^{\Z}}$ such that $\sigma=\sigma^{\Z}\otimes_{\Z}k$.
		\item\label{it:module_topological} For any $M$ in $\mod{\Lambda}$, $M\cong M^{\Z}\otimes_{\Z}k$ for some $\Lambda^{\Z}$-module $M^{\Z}$ which is free as a $\Z$-module.
	\end{enumerate}
	Then, the pretriangulated category structure on $\proj{\Lambda}$ with inverse shift functor $M\mapsto M_{\sigma}$ is neither an algebraic triangulated category structure over any commutative ring nor a topological triangulated category structure.
\end{theorem}

The proof of \Cref{thm:non-topological}, at the very end of this section, is based on some cohomology computations. We now apply this \namecref{thm:non-topological} to the example in \Cref{sec:example}.

\begin{corollary}\label{cor:non-topological}
	Let $k$ be any algebraically closed field. The category $\proj{P(\L{2})}$ has a pretriangulated category structure with inverse shift functor $M\mapsto M_\sigma$, for $\sigma=\nu\phi$ the composition of the automorphisms in \eqref{eq:nu_automorphism} and \Cref{prop:exotic_automorphism}, which is neither an algebraic triangulated category structure over any commutative ring nor a topological triangulated category structure.
\end{corollary}

\begin{proof}
	The hypotheses of \Cref{thm:non-algebraic} were checked in the previous section, culminating with \Cref{cor:non-algebraic}.

	We write $P_{\Z}(\L{2})$ for the preprojective algebra of generalized Dynkin type $\L{2}$ over the integers, presented by \eqref{eq:presentation}. It is $\Z$-free with the basis indicated in \eqref{eq:basis}. Moreover, the automorphisms $\nu$ and $\phi$ in \eqref{eq:nu_automorphism} and \Cref{prop:exotic_automorphism} are well-defined automorphisms of $P_{\Z}(\L{2})$. We denote them by $\nu^\Z$ and $\phi^\Z$ to indicate that the ground ring is $\Z$, and we also denote $\sigma^\Z=\nu^\Z\phi^\Z$.

	The hypotheses of \Cref{thm:non-topological} also hold. More precisely, \eqref{it:algebra_topological} holds because $P(\L{2})\cong P_{\Z}(\L{2})\otimes_{\Z}k$ and \eqref{it:automorphism_topological} because $\sigma=\sigma^\Z\otimes_{\Z}k$.

	Let us finally check \eqref{it:module_topological}. 
The 18 matrices in the 6 minimal exact triangles in \eqref{eq:triangles} are precisely the minimal presentations of the 18 indecomposable non-projective $P(\L{2})$-modules. This is easily checked by inspection using the stable Auslander–Reiten quiver in the proof of \Cref{prop:exotic_automorphism}. Those matrices have entries in $P_{\Z}(\L{2})$. The minimal exact triangles remain twisted $3$-periodic long exact sequences replacing $P^i=e_iP(\L{2})$ with $e_iP_{\Z}(\L{2})$, $i=1,2$. Therefore, the matrices present appropriate $P_{\Z}(\L{2})$-modules $M^\mathbb{Z}$ for the indecomposable non-projective $P(\L{2})$-modules $M$. In particular, $M^\mathbb{Z}$ is $\Z$-free because it is the kernel of a morphism between projective $P_{\Z}(\L{2})$-modules and $P_{\Z}(\L{2})$ is $\Z$-free.
\end{proof}

We now recall some basic facts on the Hochschild(--Mitchell) cohomology of $k$-algebras, and more generally $k$-linear categories, see e.g.~\cite{cartan_eilenberg_1956_homological_algebra,mitchell_1972_rings_several_objects}.

\begin{definition}\label{def:HH}
	Let $k$ be a commutative ring and $\A$ a $k$-linear category. The \emph{bar complex} $\B{\star}{\A}{}$ is the simplicial $\A$-bimodule with
	\[\B{n}{\A}{}=\bigoplus_{x_0,\dots,x_n}\A(x_0,-)\otimes\A(x_1,x_0)\otimes\cdots\otimes\A(x_n,x_{n-1})\otimes\A(-,x_n).\]
	This direct sum is indexed by all sequences of $n+1$ objects in $\A$.
	Face operators are
	\[d_i(f_0\otimes\cdots\otimes f_{n+1})=\cdots\otimes f_{i}f_{i+1}\otimes\cdots,\quad n\geq 1,\quad 0\leq i\leq n,\]
	and degeneracies are
	\[s_i(f_0\otimes\cdots\otimes f_{n+1})=\cdots\otimes f_i\otimes \id{x_i}\otimes f_{i+1}\otimes\cdots,\quad n\geq 0, \quad 0\leq i\leq n.\]
	Here $f_i\colon x_i\to x_{i-1}$.

	The bar complex is regarded as a chain complex in the usual way, defining the differential as $d=\sum_{i=0}^{n}(-1)^i d_i$ in each degree $n\geq 1$. It is a resolution of $\A$ as an $\A$-bimodule (projective if $\A$ is locally $k$-projective, e.g.~if $k$ is a field). We will also consider the homotopy equivalent \emph{normalized bar complex} $\nB{\star}{\A}{}$, which is the chain complex obtained by modding out the image of degeneracies in $\B{\star}{\A}{}$.

	Let $M$ be an $\A$-bimodule. The \emph{Hochschild complex} of $\A$ with coefficients in $M$ is
	\[\HC{\star}{\A,M}{}=\hom_{\A^e}(\B{\star}{\A}{},M).\]
	Degree-wise, we can identify
	\[\HC{n}{\A,M}{}=\prod_{x_0,\dots,x_n}\hom_k(\A(x_1,x_0)\otimes\cdots\otimes\A(x_n,x_{n-1}), M(x_n,x_0)).\]
	Hence, a \emph{Hochschild cochain} $c_n\in \HC{n}{\A,M}{}$ is a multilinear map which can be evaluated on sequences of $n$ composable morphisms in $\A$, i.e.~$c_n(f_1,\dots,f_n)$ or equivalently
	\[c_n(x_0\xleftarrow{f_1}x_1\leftarrow\cdots\leftarrow x_{n-1 }\xleftarrow{f_n}x_{n})\in M(x_n,x_0),\]

	The \emph{Hochschild cohomology} of $\A$ with coefficients in $M$ is the cohomology of the Hochschild complex, denoted by \[\HH{\star}{\A,M}{}\]
	We can also compute it as the cohomology of the \emph{normalized Hochschild complex}
	\[\nHC{\star}{\A,M}=\hom_{\A^e}(\nB{\star}{\A}{},M).\]
	With the previous identification, $\nHC{\star}{\A,M}\subset\HC{\star}{\A,M}{}$ is the subcomplex consisting of the Hochschild cochains which vanish when one entry is an identity morphism in $\A$. This condition is void in degree $0$. Hence, we say that cochains in $\nHC{n}{\A,M}$ are \emph{normalized with respect to identities}.

	When $M=\A$, the Hochschild complex $\HC{\star}{\A,\A}{}$ is a differential graded algebra with the \emph{cup product} defined as
	\[(c_p\smile c_q)(f_1,\dots,f_{p+q})=c_p(f_1,\dots,f_p)c_q(f_{p+1},\dots,f_{p+q}).\]
	Hence $\HH{\star}{\A,\A}{}$ is a graded algebra (actually graded commutative). Moreover, for $M$ an arbitrary $\A$-bimodule, $\HC{\star}{\A,M}{}$ is a differential graded $\HC{\star}{\A,\A}{}$-bimodule with the \emph{cup product} defined by the same formula, so $\HH{\star}{\A,M}{}$ is a graded $\HH{\star}{\A,\A}{}$-bimodule (not necessarily symmetric).

	When we want to stress the ground ring we are considering, we will add it as a subscript, e.g.~$\HC{\star}{\A,M}{k}$, $\HH{\star}{\A,M}{k}$. If $\ell\subset k$ is a subring then $\HC{\star}{\A,M}{k}\subset \HC{\star}{\A,M}{\ell}$ is the subcomplex of $k$-linear cochains. This inclusion induces a \emph{restriction of scalars} morphism in Hochschild cohomology
	\begin{equation}\label{eq:restriction_of_scalars}
		\HH{\star}{\A,M}{k}\longrightarrow\HH{\star}{\A,M}{\ell}.
	\end{equation}

	The forgetful functor from $k$-linear categories to ordinary categories has a left adjoint, sending an ordinary category $\C$ to the $k$-linear category $k\C$ with the same objects as $\C$ such that the morphism $k$-module $k\C(x,y)$ has basis $\C(x,y)$. If $M$ is a $k\C$-bimodule, the Hochschild complex $\HC{\star}{k\C,M}{}$ is given by
	\[\HC{n}{k\C,M}{}=\prod_{x_0,\dots,x_n}\Set(\C(x_1,x_0)\times\cdots\times\C(x_n,x_{n-1}), M(x_n,x_0))\]
	where $\Set$ is the category of sets and maps between them.

	If $\A$ is a $k$-linear category, we can still consider $k\A$. The Hochschild complex $\HC{\star}{k\A,\A}{}$ is a differential graded algebra with cup product defined as above and, if $M$ is an $\A$-bimodule, $\HC{\star}{k\A,M}{}$ is a differential graded $\HC{\star}{k\A,\A}{}$-bimodule. In particular, the Hochschild cohomology $\HH{\star}{k\A,\A}{}$ is a graded algebra and $\HH{\star}{k\A,M}{}$ is a graded $\HH{\star}{k\A,\A}{}$-bimodule.

	If $M$ is an $\A$-bimodule, the obvious counit $k$-linear functor $k\A\to\A$ induces an inclusion of Hochschild complexes $\HC{\star}{\A,M}{}\subset \HC{\star}{k\A,M}{}$. The small one consists of the multilinear cochains in the big one. This inclusion induces a \emph{comparison} morphism in Hochschild cohomology
	\begin{equation}\label{eq:comparison}
		\HH{\star}{\A,M}{}\longrightarrow\HH{\star}{k\A,M}{}.
	\end{equation}
\end{definition}

\begin{remark}\label{rem:HH_Morita_and_cocycle}
	Let $A$ be a $k$-algebra and $M$ an $A$-bimodule. By Morita invariance, the inclusion $A\hookrightarrow\proj{A}$ of the free $A$-module of rank $1$ in the category of finitely generated projective $A$-modules induces an isomorphism
	\[\HH{\star}{\proj{A},\hom_A(-,-\otimes_AM)}{}\cong \HH{\star}{A,M}{}.\]

	We record for later use that the cocycle condition for
	\[c\in\HC{1}{\proj{A},\hom_A(-,-\otimes_AM)}{}\]
	at a pair $(f_1,f_2)$ of composable morphisms in $\proj{A}$
	is
	\[c(f_1f_2)=(f_1\otimes_AM)c(f_2)+c(f_1)f_2\]
	for any pair of composable morphisms in $\proj{A}$,
	and the same formula holds after replacing $\proj{A}$ with $k\proj{A}$ in the Hochschild complex.
\end{remark}

\begin{definition}\label{def:edge_HH}
	Let $k$ be a field and $A$ a coherent $k$-algebra.
	Denote by $\proj{A}$ the category of finitely generated projective $A$-modules. Given an $A$-bimodule $M$ and a finitely presented $A$-module $N$, the \emph{edge morphism}
	\[\edge{N}\colon\HH{n}{A,M}{}=\HH{n}{\proj{A},\hom_A(-,-\otimes_AM)}{}\longrightarrow\ext_A^n(N,N\otimes_AM)\]
	is defined as follows. Choose a projective resolution of $N$ in $\mod{A}$, which is abelian by coherence,
	\[\cdots\to P_{i+1}\xrightarrow{f_{i+1}}P_{i}\xrightarrow{f_{i}}P_{i-1}\to\cdots\to P_0\stackrel{\pi}{\twoheadrightarrow} N.\]
	Given a Hochschild cohomology class $[c]$ in the source represented by a cocycle $c$ in the Hochschild complex of $\proj{A}$, the element $\edge{N}([c])\in\ext^n_A(N,N\otimes_A M)$ is represented by
	\[P_n\xrightarrow{c(f_1,\dots,f_n)} P_0\otimes_AM\xrightarrow{\pi\otimes_AM}N\otimes_AM.\]
\end{definition}

\begin{remark}
	The collection of edge morphisms $\edge{N}$ for all finitely presented $A$-modules $N$ assembles to an honest edge morphism of the second spectral sequence in \cite[Remark 5.4]{muro_2020_first_obstructions_enhancing} (ungraded version), compare \cite[(5.3)]{muro_2020_first_obstructions_enhancing}. In particular, they are well defined.
\end{remark}

\begin{proposition}\label{prop:HH_field_extensions}
	Let $\ell\subset k$ be a field extension, $A$ an $\ell$-algebra and $M$ an $A$-bimodule.
	\begin{enumerate}
		\item\label{it:field_extension_over_big} If $\dim_{\ell}A<\infty$ then $\HH{\star}{A\otimes_{\ell}k,M\otimes_{\ell}k}{k}\cong\HH{\star}{A,M}{\ell}\otimes_{\ell} k$ over $k$,
		\item\label{it:field_extension_over_small} If $\dim_{\ell}A<\infty$ and $\dim_{\ell}M<\infty$ then $\HH{\star}{A\otimes_{\ell}k,M\otimes_{\ell}k}{\ell}\cong\HH{\star}{A,M}{\ell}\otimes_{\ell}\HH{\star}{k,k}{\ell}$ over $\ell$.
	\end{enumerate}
	Under the assumptions of \eqref{it:field_extension_over_small}, and hence of \eqref{it:field_extension_over_big}, the restriction of scalars morphism in~\eqref{eq:restriction_of_scalars}
	\[\HH{\star}{A\otimes_{\ell}k,M\otimes_{\ell}k}{k}\longrightarrow
	\HH{\star}{A\otimes_{\ell}k,M\otimes_{\ell}k}{\ell}\]
	is the tensor product over $\ell$ of $\HH{\star}{A,M}{\ell}$ with the inclusion of $\HH{0}{k,k}{\ell}=k$ in $\HH{\star}{k,k}{\ell}$.
\end{proposition}

\begin{proof}
	Let us first prove \eqref{it:field_extension_over_big}. We have an isomorphism of complexes
	\[\HC{\star}{A\otimes_{\ell}k,M\otimes_{\ell}k}{k}\cong\HC{\star}{A,M}{\ell}\otimes_{\ell}k\]
	given by
	\begin{align*}
		\hom_{k}((A\otimes_{\ell}k)^{\otimes_kn},M\otimes_{\ell}k)&\cong \hom_{\ell}(A^{\otimes_{\ell}n},M\otimes_{\ell}k)\\
		&\cong \hom_{\ell}(A^{\otimes_{\ell}n},M)\otimes_{\ell}k.
	\end{align*}
	Here, the first isomorphism is given by the extension and restriction of scalars adjunction, and the second one holds by finite-dimensionality of $A$ over $\ell$.

	We now prove \eqref{it:field_extension_over_small}. Over $\ell$, $\B{\star}{A\otimes_{\ell}k}{}$ and $\B{\star}{A}{}\otimes_{\ell}\B{\star}{k}{}$ are both projective resolutions of $A\otimes_{\ell}k$ as a bimodule over itself. Therefore, they are homotopy equivalent. An explicit homotopy equivalence can be found in \cite[XI.6, $g$ on p.~218--219]{cartan_eilenberg_1956_homological_algebra}. We can therefore use both resolutions to compute $\HH{\star}{A\otimes_{\ell}k,M\otimes_{\ell}k}{\ell}$. Since $A$ and $M$ are finite-dimensional over $\ell$,
	\begin{align*}
		\hom_{(A\otimes_{\ell}k)^e}(\B{\star}{A}{}\otimes_{\ell}\B{\star}{k}{},M\otimes_{\ell}k)&\cong \hom_{A^e}(\B{\star}{A}{},M)\otimes_{\ell}\hom_{k^e}(\B{\star}{k}{},k)\\
		&=\HC{\star}{A,M}{\ell}\otimes_{\ell}\HC{\star}{k,k}{\ell}.
	\end{align*}
	Bidegree-wise, this isomorphism is given by
	\begin{align*}
		\hom_{\ell}(A^{\otimes_{\ell} p}\otimes_{\ell} k^{\otimes_{\ell} q},M\otimes_{\ell} k)&\cong \hom_{\ell}(A^{\otimes_{\ell} p},M)\otimes_{\ell} \hom_{\ell}(k^{\otimes_{\ell} q},k).
	\end{align*}
	For later use, we record here that a direct homotopy equivalence
	\[\HC{\star}{A,M}{\ell}\otimes_{\ell}\HC{\star}{k,k}{\ell}\longrightarrow
		\HC{\star}{A\otimes_{\ell}k,M\otimes_{\ell}k}{\ell}\]
	is given by the maps
	\[\rho_{p,q}\colon\HC{p}{A,M}{\ell}\otimes_{\ell}\HC{q}{k,k}{\ell}\longrightarrow
		\HC{p+q}{A\otimes_{\ell}k,M\otimes_{\ell}k}{\ell}\]
	defined by
	\begin{multline*}\rho_{p,q}(c_p,c_q)(a_1\otimes x_1,\dots,a_{p+q}\otimes x_{p+q})=\\
	c_p(a_1,\dots,a_p)a_{p+1}\cdots a_{p+q}\otimes x_1\cdots x_pc_q(x_{p+1},\dots,x_{p+q}).\end{multline*}

	The final claim is clear from the preceding computations.
\end{proof}

\begin{proposition}\label{prop:HH_edge_vanishing}
	Let $\ell\subset k$ be a field extension, $A$ an $\ell$-algebra, $M$ an $A$-bimodule and $N$ an $A$-module. Suppose that $\dim_{\ell}A,\dim_{\ell}M,\dim_{\ell}N<\infty$. Then, the edge morphism in \Cref{def:edge_HH}
	\[\edge{N\otimes_{\ell}k}\colon\HH{\star}{A \otimes_{\ell}k,M \otimes_{\ell}k}{\ell}\longrightarrow\ext_{A \otimes_{\ell}k}^{\star}(N \otimes_{\ell}k,(N \otimes_{\ell}k)\otimes_{A \otimes_{\ell}k }(M \otimes_{\ell}k))\]
vanishes on the second factor of the direct sum decomposition
	\[\HH{\star}{A\otimes_{\ell}k,M\otimes_{\ell}k}{\ell}\cong\HH{\star}{A,M}{\ell}\otimes_{\ell}k\oplus \HH{\star}{A,M}{\ell}\otimes_{\ell}\HH{>0}{k,k}{\ell}\]
	derived from \Cref{prop:HH_field_extensions} \eqref{it:field_extension_over_small}.
\end{proposition}

\begin{proof}
	Choose a projective resolution of $N$ in $\mod{A}$,
	\[\cdots\to P_{i+1}\xrightarrow{f_{i+1}}P_{i}\xrightarrow{f_{i}}P_{i-1}\to\cdots\to P_0\stackrel{\pi}{\twoheadrightarrow} N.\]
	Then
	\[\cdots\to P_{i+1} \otimes_{\ell}k\xrightarrow{f_{i+1}\otimes 1}P_{i} \otimes_{\ell}k\xrightarrow{f_{i} \otimes 1}P_{i-1} \otimes_{\ell}k\to\cdots\to P_0 \otimes_{\ell}k\stackrel{\pi \otimes 1}{\twoheadrightarrow} N \otimes_{\ell}k\]
	is a projective resolution of $N \otimes_{\ell}k $ in $\mod{A \otimes_{\ell}k}$.

	By the proof of \Cref{prop:HH_field_extensions}, any element in $\HH{\star}{A\otimes_{\ell}k,M\otimes_{\ell}k}{\ell}$ coming from the direct factor $\HH{\star}{A,M}{\ell}\otimes_{\ell}\HH{>0}{k,k}{\ell}$ is a linear combination of classes represented by cocycles of the form $\rho_{p,q}(c_p,c_q)$ for certain cocycles $c_p\in\HC{p}{A,M}{\ell}$, $c_q\in\HC{q}{k,k}{\ell}$ and $q>0$. We can assume without loss of generality that $c_q$ is normalized with respect to identities. Therefore,
	\[\rho_{p,q}(c_p,c_q)(f_1\otimes 1,\dots,f_{p+q}\otimes 1)
	=c_p(f_1,\dots,f_p)f_{p+1}\cdots f_{p+q}\otimes c_q(1,\dots 1)=0.\]
	Here we use that $q>0$, so there is at least one $1$ among the arguments of $c_q$ and $c_q(1,\dots,1)=0$ by normalization. This proves that $\edge{N\otimes_{\ell}k}([\rho_{p,q}(c_p,c_q)])=0$.
\end{proof}

The following characterization of the MacLane cohomology is well known, see \cite[Proposition 3.12]{jibladze_pirashvili_1991_cohomology_algebraic_theories} and \cite[Theorem 1.4]{pirashvili_waldhausen_1992_mac_lane_homology}.

\begin{definition}
	Let $A$ be a ring and $M$ an $A$-bimodule. The \emph{MacLane cohomology} of $A$ with coefficients in $M$ is defined as
	\[\HML{\star}{A,M}=\HH{\star}{\Z\proj{A},\hom_A(-,-\otimes_AM)}{}.\]
	If $A$ is an algebra over a commutative ring $k$, we can replace $\Z$ with $k$ obtaining the same result.
\end{definition}

The comparison morphism \eqref{eq:comparison} for the category $\A=\proj{A}$ and its bimodule $\hom_A(-,-\otimes_AM)$ is of the form
\begin{equation}\label{eq:comparison_HH_HML}
	\HH{\star}{A,M}{}\longrightarrow\HML{\star}{A,M}.
\end{equation}

\begin{proposition}\label{prop:comparison_SH_HML_char_0}
	If $A$ is a $\Q$-algebra and $M$ is an $A$-bimodule, the comparison morphism $\HH{\star}{A,M}{}\to\HML{\star}{A,M}$ in \eqref{eq:comparison_HH_HML} is an isomorphism.
\end{proposition}

\begin{proof}
	The cohomological version of the spectral sequence in \cite[Theorem 4.1]{pirashvili_waldhausen_1992_mac_lane_homology} looks like
	\[E_2^{p,q}=\HH{p}{A,\HML{q}{\Z,M}}{}\Longrightarrow\HML{p+q}{A,M},\]
	compare \cite[9.2.1 Theorem]{baues_pirashvili_2006_comparison_mac_lane}.
	The comparison morphism in the statement is the edge morphism of this spectral sequence. It therefore suffices to show that $\HML{\star}{\Z,M}=M$ concentrated in degree $0$. In order to prove this, we use the universal coefficient spectral sequence for the computation of MacLane cohomology from homology,
	\[E^{p,q}_2=\ext^{p}_{\Z}(\HMLhomology{q}{\Z,\Z},M)\Longrightarrow\HML{p+q}{\Z,M}.\]
	This spectral sequence is constructed in \cite[Proposition 8.2]{muro_2020_first_obstructions_enhancing}. The MacLane homology $\HMLhomology{\star}{\Z,\Z}$ is well known to be $\Z$ in degree $0$ and torsion in higher degrees \cite{bokstedt_1985_topological_hochschild_homology}. Therefore $\HML{\star}{\Z,M}=M$ concentrated in degree $0$, since $M$ is torsion-free and injective as an abelian group.
\end{proof}

\begin{definition}\label{def:MacLane_cohomology_class}
	Let $p>0$ be a prime and $A$ a $\Z/(p^2)$-algebra which is free as a $\Z/(p^2)$-module. We define a MacLane cohomology class
	\[[\xi_A]\in\HML{2}{A/(p),A/(p)}=\HH{2}{\Fp{p}\proj{A/(p)},\hom_{A/(p)}}{}\]
	in the following way. The functor
	\[-\otimes_AA/(p)\colon\proj{A}\longrightarrow\proj{A/(p)},\qquad P \otimes_AA/(p)=P/(p),\] is full, conservative and essentially surjective. Each $P$ in $\proj{A}$ is $\Z/(p^2)$-free and fits into a natural short exact sequence of $A$-modules
	\begin{equation}\label{eq:natural_short_exact_sequence}
		P/(p)\stackrel{j_P}\hookrightarrow P \stackrel{q_P}\twoheadrightarrow P/(p)
	\end{equation}
	where $q_P$ is the canonical projection and $j_P([x])=px$.
	For any pair of objects $P,Q$ in $\proj{A}$ we have a short exact sequence
	\[
	\begin{tikzcd}
	\hom_{A/(p)}(P/(p),Q/(p))\ar[r, hook]&
		\hom_{A}(P,Q)\ar[two heads, rr, "-\otimes_AA/(p)"]&&	\hom_{A/(p)}(P/(p),Q/(p))
		\ar[ll, bend left, hook', dashed, "s_{P,Q}"]
	\end{tikzcd}
	\]
	The first arrow is defined by $f\mapsto j_Qfq_P$. We choose a splitting map of sets $s_{P,Q}$. A representing cocycle $\xi_A$ is defined as follows in terms of these splittings. Given $P_0,P_1,P_2$ in $\proj{A}$ and
	\[P_2/(p)\xrightarrow{f_2} P_1/(p)\xrightarrow{f_1}P_0/(p)\]
	in $\proj{A/(p)}$,
	\[s_{P_2,P_0}(f_1f_2)-s_{P_1,P_0}(f_1) s_{P_2,P_1}(f_2)=j_{P_0}\xi_A(f_1,f_2)q_{P_2}.\]
	We will assume without loss of generality that the splittings $s_{P,Q}$ are pointed, i.e.~$s_{P,Q}(0)=0$ for all $P,Q$ in $\proj{A}$. Hence $\xi_A(f_1,f_2)=0$ if $f_1=0$ or $f_2=0$, i.e.~$\xi_A$ is \emph{zero-normalized} in the sense of \cite[(1.6) Example]{baues_dreckmann_1989_cohomology_homotopy_categories}.
\end{definition}

\begin{proposition}\label{prop:direct_sum_HML}
	Let $p>0$ be a prime, $A$ a $\Z/(p^2)$-algebra which is free as a $\Z/(p^2)$-module and $M$ an $A/(p)$-bimodule. We have a direct sum decomposition
	\[\HML{3}{A/(p),M}=\HH{3}{A/(p),M}{}\oplus \HML{1}{A/(p),M}{}\]
	defined by the morphisms
	\[\HH{3}{A/(p),M}{}\xrightarrow{\eqref{eq:comparison_HH_HML}}\HML{3}{A/(p),M}\xleftarrow{-\smile[\xi_A]}\HML{1}{A/(p),M}.\]
\end{proposition}

\begin{proof}
	\citeauthor{baues_2006_algebra_secondary_cohomology} constructed in \cite[4.6.1 Definition and 4.6.3 Proposition]{baues_2006_algebra_secondary_cohomology} a morphism $\Gamma_p$ such that the composite of the two maps in the sequence
	\begin{equation}\label{eq:exact_sequence_HML}
		\HH{3}{A/(p),M}{}\xrightarrow{\eqref{eq:comparison_HH_HML}}\HML{3}{A/(p),M}\xrightarrow{\Gamma_p}\HML{1}{A/(p),M}{}
	\end{equation}
	is zero.

	\citeauthor{baues_pirashvili_2006_comparison_mac_lane} proved in \cite[7.4.1.~Theorem]{baues_pirashvili_2006_comparison_mac_lane} that \eqref{eq:exact_sequence_HML} is exact, after identifying MacLane cohomology with Shukla cohomology in degrees $2$ and $3$, see \cite[\S9.1 (10) and 9.1.1 Proposition]{baues_pirashvili_2006_comparison_mac_lane}. Here we use that $A/(p)$ is an $\Fp{p}$-algebra.
	Moreover, the first map in \eqref{eq:exact_sequence_HML} is injective by \cite[7.3.1 Theorem]{baues_pirashvili_2006_comparison_mac_lane} since $A$ is a weak lifting of $A/(p)$ in the sense of \cite[\S7.2]{baues_pirashvili_2006_comparison_mac_lane}. Therefore, it suffices to check that the map
	\[\HML{1}{A/(p),M}\xrightarrow{-\smile[\xi_A]}\HML{3}{A/(p),M}\]
	splits $\Gamma_p$ up to sign.

	At the cochain level,
	\begin{align*}
		\Gamma_p(c_3)(x_0\xleftarrow{f_1}x_1)={}&c_3(x_0\xleftarrow{f_1}x_1\xleftarrow{(1\dots1)}x_1\oplus\stackrel{p}{\cdots}\oplus x_1\xleftarrow{\begin{psmallmatrix*}[c]
			1\\\svdots\\1
		\end{psmallmatrix*}}x_1)\\
		&-c_3(x_0\xleftarrow{(1\dots1)}x_0\oplus\stackrel{p}{\cdots}\oplus x_0\xleftarrow{f_1\oplus\stackrel{p}\cdots\oplus f_1}x_1\oplus\stackrel{p}{\cdots}\oplus x_1\xleftarrow{\begin{psmallmatrix*}[c]
			1\\\svdots\\1
		\end{psmallmatrix*}}x_1)\\
		&+c_3(x_0\xleftarrow{(1\dots1)}x_0\oplus\stackrel{p}{\cdots}\oplus x_0\xleftarrow{\begin{psmallmatrix*}[c]
			1\\\svdots\\1
		\end{psmallmatrix*}}x_0\xleftarrow{f_1}x_1).
	\end{align*}
	Therefore,
	\begin{multline*}
		\Gamma_p(c_1\smile\xi_A)(x_0\xleftarrow{f_1}x_1)=c_1(x_0\xleftarrow{f_1}x_1)\xi_A(x_1\xleftarrow{(1\dots1)}x_1\oplus\stackrel{p}{\cdots}\oplus x_1\xleftarrow{\begin{psmallmatrix*}[c]
			1\\\svdots\\1
		\end{psmallmatrix*}}x_1)\\
		-c_1(x_0\xleftarrow{(1\dots1)}x_0\oplus\stackrel{p}{\cdots}\oplus x_0)\xi_A(x_0\oplus\stackrel{p}{\cdots}\oplus x_0\xleftarrow{f_1\oplus\stackrel{p}\cdots\oplus f_1}x_1\oplus\stackrel{p}{\cdots}\oplus x_1\xleftarrow{\begin{psmallmatrix*}[c]
			1\\\svdots\\1
		\end{psmallmatrix*}}x_1)\\
		+c_1(x_0\xleftarrow{(1\dots1)}x_0\oplus\stackrel{p}{\cdots}\oplus x_0)\xi_A(x_0\oplus\stackrel{p}{\cdots}\oplus x_0\xleftarrow{\begin{psmallmatrix*}[c]
			1\\\svdots\\1
		\end{psmallmatrix*}}x_0\xleftarrow{f_1}x_1).
	\end{multline*}
	We can suppose that $c_1$ is normalized with respect to identities in the sense of \Cref{def:HH}, i.e.~$c_1(1)=0$. Moreover, we can \emph{also} suppose that it respects sum diagrams in the sense of \cite[Definition A.3]{baues_tonks_1996_sumnormalised_cohomology_categories} by \cite[Theorem A.9]{baues_tonks_1996_sumnormalised_cohomology_categories} (the cochain $c'=c+\delta\gamma_c$ in the proof of that result, which respects sum diagrams, is normalized with respect to identities if $c$ is). Therefore,
	\[\begin{array}{ccccc}
		&&c_1(x_0\xleftarrow{(1\dots1)}x_0^p)&\in&\hom_A(x_0^p,x_0\otimes_{A/(p)}M)\\
		&&\|&&\|\\
		(0,\dots,0)&=&(c_1(1),\dots,c_1(1))&\in&\hom_A(x_0,x_0\otimes_{A/(p)}M)^p
	\end{array}\]
	and hence
	\[\Gamma_p(c_1\smile\xi_A)(x_0\xleftarrow{f_1}x_1)=c_1(x_0\xleftarrow{f_1}x_1)\xi_A(x_1\xleftarrow{(1\dots1)}x_1\oplus\stackrel{p}{\cdots}\oplus x_1\xleftarrow{\begin{psmallmatrix*}[c]
			1\\\svdots\\1
		\end{psmallmatrix*}}x_1).\]
	To conclude, we show that the factor on the right of the product is $-1\colon x_1\to x_1$.

	The object $x_1$ belongs to $\proj{A/(p)}$ so $x_1=P/(p)$ for some $P$ in $\proj{A}$. We can suppose that the splitting maps $s$ in \Cref{def:MacLane_cohomology_class} send any matrix of $0$'s and $1$'s to the same matrix, where $1$ in the source (resp.~target) of $s$ denotes an identity morphism in $\proj{A/(p)}$ (resp.~$\proj{A}$).
	Hence,
	\begin{multline*}
		s_{P,P}(\underbrace{x_1\xleftarrow{(1\dots1)}x_1\oplus\stackrel{p}{\cdots}\oplus x_1\xleftarrow{\begin{psmallmatrix*}[c]
			1\\\svdots\\1
		\end{psmallmatrix*}}x_1}_{p\cdot 1=0\colon x_1\to x_1})\\
		-
		s_{P\oplus\stackrel{p}{\cdots}\oplus P,P}(x_1\xleftarrow{(1\dots1)}x_1\oplus\stackrel{p}{\cdots}\oplus x_1)s_{P,P\oplus\stackrel{p}{\cdots}\oplus P}(x_1\oplus\stackrel{p}{\cdots}\oplus x_1\xleftarrow{\begin{psmallmatrix*}[c]
			1\\\svdots\\1
		\end{psmallmatrix*}}x_1)\\
		=(P\xleftarrow{0}P)-(P\xleftarrow{(1\dots1)}P\oplus\stackrel{p}{\cdots}\oplus P\xleftarrow{\begin{psmallmatrix*}[c]
			1\\\svdots\\1
		\end{psmallmatrix*}}P)\\
		=-(P\xleftarrow{p\cdot 1}P)=P\xleftarrow{j_P}x_1\xleftarrow{-1}x_1\xleftarrow{q_P}P.
	\end{multline*}
	This means that
	\[\xi_A(x_1\xleftarrow{(1\dots1)}x_1\oplus\stackrel{p}{\cdots}\oplus x_1\xleftarrow{\begin{psmallmatrix*}[c]
			1\\\svdots\\1
		\end{psmallmatrix*}}x_1)=x_1\xleftarrow{-1}x_1.\]
\end{proof}

\begin{definition}\label{def:edge_HML}
	Let $A$ be a coherent $k$-algebra.
	Given an $A$-bimodule $M$ and a finitely presented $A$-module $N$, the \emph{edge morphism}
	\[\edge{N}\colon\HML{n}{A,M}=\HH{n}{k\proj{A},\hom_A(-,-\otimes_AM)}{}\longrightarrow\ext_A^n(N,N\otimes_AM)\]
	is defined by exactly the same formula as in \Cref{def:edge_HH}. This makes sense because, in \Cref{def:edge_HH}, we did not use that the Hochschild cocycle $c$ is multilinear.
\end{definition}

\begin{proposition}
	Let $p>0$ be a prime, $A$ a $\Z/(p^2)$-algebra which is free as a $\Z/(p^2)$-module, $M$ an $A/(p)$-bimodule and $N$ an $A$-module which is free as a $\Z/(p^2)$-module. Then, the morphism in \Cref{def:edge_HML}
	\[\edge{N/(p)}\colon\HML{3}{A/(p),M}\longrightarrow\ext_{A/(p)}^3(N/(p),N/(p)\otimes_{A/(p)}M)\]
	vanishes on the second direct factor of the direct sum decomposition in \Cref{prop:direct_sum_HML}
	\[\HML{3}{A/(p),M}=\HH{3}{A/(p),M}{}\oplus \HML{1}{A/(p),M}{}.\]
\end{proposition}

\begin{proof}
	Choose a projective resolution of $N$ in $\mod{A}$,
	\[\cdots\to P_{i+1}\xrightarrow{\tilde{f}_{i+1}}P_{i}\xrightarrow{\tilde{f}_{i}}P_{i-1}\to\cdots\to P_0\stackrel{\tilde{\pi}}{\twoheadrightarrow} N.\]
	Since $A$ and $N$ are free as $\Z/(p^2)$-modules, modding out $(p)$ we obtain a projective resolution of $N/(p)$ in $\mod{A/(p)}$,
	\[\cdots\to P_{i+1}/(p)\xrightarrow{f_{i+1}}P_{i}/(p)\xrightarrow{f_{i}}P_{i-1}/(p)\to\cdots\to P_0/(p)\stackrel{\pi}{\twoheadrightarrow} N/(p).\]
	A class $[c_1]\in \HML{1}{A/(p),M}{}$ identifies with $[c_1]\smile[\xi_A]\in\HML{3}{A/(p),M}$. Its image $\edge{N/(p)}([c_1]\smile[\xi_A])$ is represented by
	\[P_3/(p)\xrightarrow{\xi_A(f_2,f_3)}P_1/(p)\xrightarrow{c_1(f_1)}P_0/(p)\otimes_{A/(p)}M\xrightarrow{\pi\otimes_{A/(p)}M}N/(p)\otimes_{A/(p)}M.\]
	Hence, we must prove that this composite factors through $f_3\colon P_3 /(p)\to P_2/(p)$.

	We start by looking at $\xi_A(f_2,f_3)$. This morphism is characterized by the following equation, see \Cref{def:MacLane_cohomology_class}:
	\begin{equation*}
		s_{P_3,P_1}(f_2f_3)-s_{P_2,P_1}(f_2)s_{P_3,P_2}(f_3)=j_{P_1}\xi_A(f_2,f_3)q_{P_3}.
	\end{equation*}
	Since $f_2f_3=0$, $s_{P_3,P_1}(f_2f_3)=0$. Moreover, both $\tilde{f}_{i+1}$ and $s_{P_{i+1},P_i}(f_{i+1})$ represent $f_{i+1}$, hence
	\[s_{P_{i+1},P_i}(f_{i+1})=\tilde{f}_{i+1}+j_{P_i}g_{i+1}q_{P_{i+1}},\qquad
	g_{i+1}\colon P_{i+1}/(p)\to P_i/(p),\qquad i=1,2.\]
	Therefore, we have
	\begin{align*}
	j_{P_1}\xi_A(f_2,f_3)q_{P_3}&=(\tilde{f}_2+j_{P_1}g_2q_{P_{2}})(\tilde{f}_3+j_{P_2}g_3q_{P_{3}})\\
	&=\tilde{f}_2\tilde{f}_3+j_{P_1}g_2q_{P_{2}}\tilde{f}_3+\tilde{f}_2j_{P_2}g_3q_{P_{3}}+j_{P_1}g_2q_{P_{2}}j_{P_2}g_3q_{P_{3}}.
	\end{align*}
	We have that $\tilde{f}_2\tilde{f}_3=0$ since they are consecutive morphisms in a projective resolution of $N$. Moreover, $q_{P_i}\tilde{f}_{i+1}=f_{i+1} q_{P_{i+1}}$ and $\tilde{f}_{i+1}j_{P_{i+1}}=j_{P_i}f_{i+1}$ for $i=1,2$, by the naturality of \eqref{eq:natural_short_exact_sequence}. Furthermore, $q_{P_{2}}j_{P_2}=0$ since \eqref{eq:natural_short_exact_sequence} is exact. Substituting these equalities in the previous equation, we obtain
	\begin{align*}
	j_{P_1}\xi_A(f_2,f_3)q_{P_3}
	&=j_{P_1}g_2f_3q_{P_{3}}+j_{P_1}f_2g_3q_{P_{3}}\\
	&=j_{P_1}(g_2f_3+f_2g_3)q_{P_{3}}.
	\end{align*}
	This means that
	\[\xi_A(f_2,f_3)=g_2f_3+f_2g_3.\]
	The composite $g_2f_3$ obviously factors through $f_3$. We will conclude this proof by showing that the composite
	\[P_3/(p)\xrightarrow{g_3}P_2/(p)\xrightarrow{f_2}P_1/(p)\xrightarrow{c_1(f_1)}P_0/(p)\otimes_{A/(p)}M\xrightarrow{\pi\otimes_{A/(p)}M}N/(p)\otimes_{A/(p)}M\]
	vanishes. Actually, the composite of the last three arrows vanishes.

	The cocycle condition for $c_1$ at $(f_1,f_2)$ says that
	\begin{align*}
		c_1(f_1f_2)&=(f_1\otimes_{A/(p)}M)c_1(f_2)+c_1(f_1)f_2,
	\end{align*}
	see \Cref{rem:HH_Morita_and_cocycle}.
	Any $1$-cocycle is zero-normalized in the sense of \Cref{def:MacLane_cohomology_class} by the cocycle condition at $(0,0)$, so $c_1(f_1f_2)=c_1(0)=0$. Therefore,
	\begin{align*}
		(\pi\otimes_{A/(p)}M)c_1(f_1)f_2&=-(\pi\otimes_{A/(p)}M)(f_1\otimes_{A/(p)}M)c_1(f_2)\\
		&=-((\pi f_1)\otimes_{A/(p)}M)c_1(f_2)=0
	\end{align*}
	since $\pi$ and $f_1$ are consecutive morphisms in a projective resolution of $N/(p)$.
\end{proof}

We are finally ready to prove the main result of this section.

\begin{proof}[Proof of \Cref{thm:non-topological}]
	We prove by contradiction that the pretriangulated structure under consideration on $\proj{\Lambda}$ is not topological. Hence we deduce that it cannot be algebraic over any commutative ring either.

	Heller's isomorphism of exact functors $\delta\colon(\Omega^{3},-\id{\Omega^{4}})\cong((-)_\sigma,\varphi_\sigma)$ is given by elements
	\[\delta(M)\in\homst_{\Lambda}(\Omega^3M,M_{\sigma})=\ext^3_{\Lambda}(M,M\otimes_{\Lambda}\Lambda_{\sigma}),\qquad M\in\mod{\Lambda}.\]
	If the pretriangulated category structure on $\proj{\Lambda}$ determined by $\delta$ is topological then, by \cite[Proposition 7.2 and Remark 7.3]{muro_2020_first_obstructions_enhancing}, all these elements are obtained, through the corresponding edge morphisms
		\[\edge{M}\colon\HML{3}{\Lambda,\Lambda_{\sigma}}\longrightarrow\ext^3_{\Lambda}(M,M\otimes_{\Lambda}\Lambda_{\sigma}),\qquad M\in\mod{\Lambda},\]
	from a single element
	\[\eta\in\HML{3}{\Lambda,\Lambda_{\sigma}}.\]
	This element deserves to be called \emph{restricted universal Toda bracket}, by comparison with the \emph{restricted universal Massey product} of \cite[\S4.5.2]{jasso_keller_muro_2023_derived_auslanderiyama_correspondence} living in Hochschild cohomology instead of MacLane cohomology.

	The notions of \emph{universal Toda bracket} and \emph{universal Massey product} date back to \cite{baues_dreckmann_1989_cohomology_homotopy_categories} and \cite{benson_krause_schwede_2004_realizability_modules_tate}.

	Let $\ell\subset k$ be the prime subfield of $k$. We denote $\Lambda^{\ell}=\Lambda^{\Z}\otimes_{\Z}\ell$ and $\sigma^{\ell}=\sigma^{\Z}\otimes_{\Z}\ell$. If $\characteristic{k}=0$ then $\ell=\Q$ and
	\begin{align*}
		\HML{3}{\Lambda,\Lambda_{\sigma}}&=\HH{3}{\Lambda,\Lambda_{\sigma}}{\Q}\\
		&=\HH{3}{\Lambda^{\Q},\Lambda^{\Q}_{\sigma^{\Q}}}{\Q}\otimes_{\Q}k\oplus\bigoplus_{i=1}^3\HH{3-i}{\Lambda^{\Q},\Lambda^{\Q}_{\sigma^{\Q}}}{\Q}\otimes_{\Q}\HH{i}{k,k}{\Q}.
	\end{align*}
	Here we use \Cref{prop:comparison_SH_HML_char_0} and \Cref{prop:HH_field_extensions} \eqref{it:field_extension_over_small}. By \Cref{prop:HH_edge_vanishing}, 
	all edge morphisms vanish on all but the first direct factor. 
	Consider the component of the restricted universal Toda bracket $\eta$ in that first direct factor,
	\[\eta_1\in \HH{3}{\Lambda^{\Q},\Lambda^{\Q}_{\sigma^{\Q}}}{\Q}\otimes_{\Q}k=\HH{3}{\Lambda,\Lambda_{\sigma}}{k}.\]
	Here we use \Cref{prop:HH_field_extensions} \eqref{it:field_extension_over_big}.
	Then, by the aforementioned vanishing, for all $M$ in $\mod{\Lambda}$, $\delta(M)$ is the image of $\eta_1$ through the edge morphism
	\[\edge{M}\colon\HH{3}{\Lambda,\Lambda_{\sigma}}{k}\longrightarrow\ext^3_{\Lambda}(M,M\otimes_{\Lambda}\Lambda_{\sigma}).\]

	If $\characteristic{k}=p>0$ then $\ell=\Fp{p}$ and
	\begin{multline*}
		\HML{3}{\Lambda,\Lambda_{\sigma}}=\HH{3}{\Lambda,\Lambda_{\sigma}}{\Fp{p}}\oplus\HML{1}{\Lambda,\Lambda_{\sigma}}=\\
		\HH{3}{\Lambda^{\Fp{p}},\Lambda^{\Fp{p}}_{\sigma^{\Fp{p}}}}{\Fp{p}}\otimes_{\Fp{p}}k\oplus\bigoplus_{i=1}^3\HH{3-i}{\Lambda^{\Fp{p}},\Lambda^{\Fp{p}}_{\sigma^{\Fp{p}}}}{\Fp{p}}\otimes_{\Fp{p}}\HH{i}{k,k}{\Fp{p}}\oplus\HML{1}{\Lambda,\Lambda_{\sigma}}.
	\end{multline*}
	Here, we use \Cref{prop:direct_sum_HML} and again \Cref{prop:HH_field_extensions} \eqref{it:field_extension_over_small}. The hypothesis of \Cref{prop:direct_sum_HML} holds because 
	\[\Lambda= \Lambda^{\Z}\otimes_{\Z} k = (\Lambda^{\Z}\otimes_{\Z} W_2(k))/(p),\]
	where $W_2(k)$ is the ring of truncated $p$-typical Witt vectors of length $2$ over the perfect field $k$. The ring $\Lambda^{\Z}\otimes_{\Z} W_2(k)$ is a $\Z/(p^2)$-algebra which is free as a $\Z/(p^2)$-module because $W_2(k)$ is and $\Lambda^{\Z}$ is a free $\Z$-module.
	Again, all edge morphisms vanish on all but the first direct factor of the latest direct sum decomposition. 
		Consider the component of the restricted universal Toda bracket $\eta$ in that first direct factor,
	\[\eta_1\in \HH{3}{\Lambda^{\Fp{p}},\Lambda^{\Fp{p}}_{\sigma^{\Fp{p}}}}{\Fp{p}}\otimes_{\Fp{p}}k=\HH{3}{\Lambda,\Lambda_{\sigma}}{k}.\]
		Here we use \Cref{prop:HH_field_extensions} \eqref{it:field_extension_over_big} once more.
		We conclude as above that, for all $M$ in $\mod{\Lambda}$, $\delta(M)$ is the image of $\eta_1$ through the edge morphism
	\[\edge{M}\colon\HH{3}{\Lambda,\Lambda_{\sigma}}{k}\longrightarrow\ext^3_{\Lambda}(M,M\otimes_{\Lambda}\Lambda_{\sigma}).\]

	In any characteristic, $\eta_1$ is an edge unit in the sense of \cite[Definition 6.4]{muro_2022_enhanced_finite_triangulated} by the reinterpretation of Heller's isomorphism of exact functors in \cite[Proposition 3.2]{muro_2020_first_obstructions_enhancing}. This implies that $\Omega^3_{\Lambda^e}(\Lambda)\cong\Lambda_{\sigma}$ in $\modst{\Lambda^e}$, see \cite[Propositions 5.7 and 6.5]{muro_2022_enhanced_finite_triangulated}. Among the hypotheses of \Cref{thm:non-algebraic}, which are also assumed here in \Cref{thm:non-topological}, we have that $\Lambda$ is basic, connected, non-separable and self-injective, so we also have an isomorphism $\Omega^3_{\Lambda^e}(\Lambda)\cong\Lambda_{\sigma}$ in $\mod{\Lambda^e}$, see the proof of \cite[Proposition 9.8]{muro_2022_enhanced_finite_triangulated}. The same argument applied to the stable bimodule isomorphism $\Omega^3_{\Lambda^e}(\Lambda)\cong\Lambda_{\nu}$ in the statement of \Cref{thm:non-algebraic} yields an isomorphism in $\mod{\Lambda^e}$. Hence $\Lambda_{\nu}\cong\Lambda_{\sigma}=\Lambda_{\nu\phi}$ in $\mod{\Lambda^e}$, so $M_{\nu}\cong M_{\nu\phi}$ naturally in $M\in\mod{\Lambda}$. Therefore $[\nu]=[\nu\phi]=[\nu][\phi]\in\Out{\Lambda}$, forcing $[\phi]$ to be trivial. This contradicts the hypothesis of \Cref{thm:non-algebraic} which says that $[\phi]$ is non-trivial.
\end{proof}

\printbibliography

\end{document}